\documentclass[11pt]{amsart}

\usepackage{fullpage}
\usepackage{color}

 \newtheorem{theorem}{Theorem}[section]
 \newtheorem{corollary}[theorem]{Corollary}
 \newtheorem{lemma}[theorem]{Lemma}
 \newtheorem{proposition}[theorem]{Proposition}
 \theoremstyle{definition}
 \newtheorem{definition}[theorem]{Definition}
 \theoremstyle{remark}
 \newtheorem{remark}[theorem]{Remark}
 \newtheorem{ex}[theorem]{Example}
 \numberwithin{equation}{section}

 \usepackage{amsmath,amsfonts,amssymb}

\usepackage{hyperref}

\usepackage{mathrsfs} 
\def \sL{\mathscr L}

\def \bC {\mathbb C}

\def \bN {\mathbb N}

\def \bR {\mathbb R}

\def \bZ {\mathbb Z}

\def \cD {\mathcal D}

\def \cF {\mathcal F}

\def \cH {\mathcal H}

\def \cL {\mathcal L}

\def \cR {\mathcal R}
\def \cS {\mathcal S}
\def \cR {\mathcal R}
\def \cR {\mathcal R}

\def \fg {\mathfrak g}

\def \fS {\mathfrak S}

\def \fU {\mathfrak U}

\def \tr {\text{\rm Tr}}

\def \id {\text{\rm I}}

\def\p#1{{\left({#1}\right)}}

\def\G{{G}}

\def \sL{\mathscr L}

\def\Gh{{\widehat{G}}}

\def\Rn{{{\mathbb R}^n}}

\def\Re{{{\rm Re}\,}}
\def\Im{{{\rm Im}\,}}

\def\Gh{\widehat G}

\begin{document}

\title{Fourier multipliers on graded Lie groups}
\author[Veronique Fischer]{Veronique Fischer}
\address{
  Veronique Fischer:
  Department of Mathematical Sciences,
  University of Bath,
  Claverton Down,
Bath  BA2 7AY, United Kingdom
  \endgraf
  {\it E-mail address} {\rm v.c.m.fischer@bath.ac.uk}
  }

\author[Michael Ruzhansky]{Michael Ruzhansky}
\address{
  Michael Ruzhansky:
  Ghent University,
Department of Mathematics,
Krijgslaan 281, Building S8,
B 9000 Ghent,
Belgium
AND
Queen Mary University of London,
School of Mathematical Sciences,
Mile End Road,
London E1 4NS,
United Kingdom
  \endgraf
  {\it E-mail address} {\rm Michael.Ruzhansky@UGent.be}
  }



\thanks{The authors were supported by the EPSRC Grant EP/K039407/1. The second authors would also like to acknowledge the subsequent support by FWO Odysseus 1 grant G.0H94.18N: Analysis and Partial Differential Equations, EPSRC grant EP/R003025/1, and LEVERHULME grant RPG-2017-151.}
      
\maketitle

\begin{abstract}
In this paper we study 
multipliers on graded nilpotent Lie groups
defined via group Fourier transform.
More precisely, 
we show that H\"ormander type conditions on the Fourier multipliers
 imply  $L^p$-boundedness.
 We express these conditions 
 using difference operators and positive Rockland operators.
 We also obtain a more refined condition using Sobolev spaces on the dual of the group which are defined and studied in this paper.
  \end{abstract}


\section{Introduction}

The Mihlin multiplier theorem \cite{mihlin, mihlin2} states that if a function $\sigma$ 
defined on $\bR^n\backslash\{0\}$ 
has at least $[d/2]+1$ continuous derivatives that satisfy 
\begin{equation}
\label{eq_mihlin_cond}
\forall \alpha\in \bN_0^d, \ |\alpha|\leq [d/2]+1,
\qquad
|\partial^\alpha\sigma (\xi)|\leq C_\alpha|\xi|^{-|\alpha|},
\end{equation}
then the Fourier multiplier operator $T_\sigma$
associated with $\sigma$, 
initially defined on Schwartz functions via
\begin{equation}
\label{eq_T_sigma}
T_\sigma \phi := \cF^{-1} \{\sigma \widehat \phi\},
\end{equation}
 admits a bounded extension on $L^p(\bR^d)$ for all $1<p<\infty$.
Above $[t]$ is the integer part of $t$
and  $\cF\phi=\widehat\phi$ denotes the Euclidean Fourier transform of a function $\phi$.
 H\"ormander improved the Mihlin multiplier theorem 
by showing  \cite{hormander}
that a sufficient condition for $T_\sigma$ to be bounded on $L^p(\bR^d)$ is the membership of $\sigma$ 
locally uniformly
to a Sobolev space $H^s(\bR^d)$ for some $s>d/2$, 
that is,
\begin{equation}
\label{eq_hormander}
\exists \eta\in \cD(0,\infty), \ \eta\not\equiv0,\qquad
\sup_{r>0} \|\sigma (r\, \cdot)\ \eta (|\cdot|^2)\|_{H^s}<\infty.
\end{equation}
If a multiplier satisfies the H\"ormander condition in \eqref{eq_hormander} with $s$ near enough $d/2$, then it satisfies the  Mihlin condition in
\eqref{eq_mihlin_cond}. 
Anisotropic analogues of the H\"ormander condition in \eqref{eq_hormander} have been studied by Rivi\`ere
\cite{riviere}.

In this paper, we present analogues of the H\"ormander and Mihlin conditions 
in the context of Lie groups equipped with (anisotropic) dilations,
and show that they imply the $L^p$-boundedness of the corresponding Fourier multiplier operators.
In the context of (unimodular type 1) Lie groups, the Fourier multipliers are defined 
formally as in \eqref{eq_T_sigma} 
but replacing the Euclidean Fourier transform with the group Fourier transform. A multiplier symbol $\sigma$ is now a field of operators parametrised by the dual  $\Gh$ of the group $G$.
Any two multiplier symbols may not necessarily commute.

The $L^p$-multiplier problem has been extensively studied in various contexts. 
On Lie groups, a large part of these studies were primarily concerned with spectral multipliers of one (or several) operator such as a sub-Laplacian, see e.g. \cite{alexopoulos, Muller+Ricci+Stein}
- with the difficult and still open question of the optimality of a Mihlin-H\"ormander condition in terms of the topological or homogeneous dimensions 
\cite{Hebish+Z,Muller+Stein,martini+muller} in the nilpotent case.
Much fewer works were devoted to Fourier multipliers.
The first study of Fourier multipliers on Lie groups
 goes back to 1971
with Coifman and Weiss' monograph \cite{coifman+weiss}
where they  developed  the Calder\'on-Zygmund theory in the setting of spaces of homogeneous types and as an application studied the Fourier multipliers of $SU(2)$, see also \cite{coifman+weiss1,coifman+weiss2}. 
But then the research into Fourier multipliers on compact Lie groups had been  focused on the  central multipliers  \cite{strichartz,vretare, weiss, weiss_SU}. 
This was so 
until the recent results on Fourier multipliers on compact Lie groups by the second author and Jens Wirth 
\cite{ruzhansky+wirth, ruzhansky+wirth1}, 
and by the first author
\cite{fischerJFA2020}.
To the authors' knowledge, the rest of the literature 
 on the  $L^p$-multiplier problem for Fourier multipliers on Lie groups 
    is restricted to the motion group
 (Rubin in 1976 \cite{rubin}) and
to the Heisenberg group stemming from the work of de Michele and Mauceri in 1979 \cite{dM+M}.

As in the latter papers \cite{ruzhansky+wirth, ruzhansky+wirth1,fischerJFA2020}, 
our hypotheses are expressed using difference operators.
The methods of proof rely on the Calder\'on-Zygmund theory adapted 
to the setting of spaces of homogeneous type as in \cite{coifman+weiss}, see also \cite{riviere}.
These methods are the classical approach for proving Fourier or spectral $L^p$-multiplier problems on nilpotent Lie groups. 
In the case of the Heisenberg group, our conditions recover and generalise the results in  
\cite{dM+M} when using the explicit description of   the difference operators from \cite[Section 6]{FR-monograph}.

Multiplier theorems and other results on nilpotent Lie groups have a wealth of applications, see \cite{rothschild+stein} for seminal results and motivation on analysis on nilpotent Lie groups, 
and \cite{Taylor84} for the case of the Heisenberg group.
Our Mihlin-H\"ormander result 
 was already used in \cite{cardona+ruzhansky}  and may lead to further advances in understanding Besov spaces and their applications.

\medskip

In this paper, we will give the analogues of both Mihlin and H\"ormander-type conditions
for the Mihlin-H\"ormander multiplier theorem. The former is given in terms of difference operators on 
the unitary dual $\Gh$ of the group $G$
which are analogues of derivatives with respect to dual variables in the case of $\Rn$.
The latter is given in terms of Sobolev spaces on $\Gh$ that we will define and study in this paper.
We will also see that 
Theorem \ref{thm_intro} under Mihlin-type conditions is implied by the H\"ormander-type condition 
of Theorem \ref{thm_main0}.
The  definitions of graded nilpotent Lie groups, homogeneous dimensions, dilations weights, Rockland operators, 
difference operators $\Delta^\alpha$,  amongst others will be recalled in Section \ref{sec_preliminary}.

\begin{theorem}
\label{thm_intro}
Let $G$ be a graded nilpotent Lie group
with homogeneous dimension $Q$.
Let  $\sigma=\{\sigma(\pi),\pi\in \Gh\}$ be
a measurable field of operators in $L^\infty(\Gh)$.
We assume that there exist a positive Rockland operator $\cR$ 
and an integer $N>Q/2$ divisible by the dilation weights
such that for all  $|\alpha|\leq N$
the following quantities are finite:
\begin{equation}
\label{eq_hyp_thm_intro}
 \sup_{\pi\in \Gh}
 \|\pi(\cR)^{\frac{[\alpha]}\nu}\Delta^\alpha\sigma\|_{\sL(\cH_\pi)}
\quad\mbox{and}\quad
 \sup_{\pi\in \Gh}
 \|\Delta^\alpha\sigma \ \pi(\cR)^{\frac{[\alpha]}\nu}\|_{\sL(\cH_\pi)},
\end{equation}
where $\nu$ is the degree of $\cR$.
Then  the Fourier multiplier operator $T_\sigma$ corresponding to $\sigma$ is bounded on $L^p(G)$ for any $1<p<\infty$.
Furthermore, 
$$
\|T_\sigma\|_{\sL(L^p(G))}
\leq C \sum_{[\alpha]\leq N}
 \p{\sup_{\pi\in \Gh}
 \|\pi(\cR)^{\frac{[\alpha]}\nu}\Delta^\alpha\sigma\|_{\sL(\cH_\pi)}
 +
  \sup_{\pi\in \Gh}
 \|\Delta^\alpha\sigma \ \pi(\cR)^{\frac{[\alpha]}\nu}\|_{\sL(\cH_\pi)}},
$$
with $C=C_{p,G}$ independent of $\sigma$.
\end{theorem}

Theorem \ref{thm_intro} applied to the abelian Euclidean setting,
that is, $(\bR^d,+)$ with the usual isotropic dilation with $\cR$ being the Laplace operator, yields the Mihlin theorem.
It will also be the case for Theorem \ref{thm_main0}.
Indeed in the Euclidean abelian setting, $\pi(\cR)$ is replaced with $|\xi|^2$ where $\xi$ is the (Fourier) dual variable.

We now give the analogue of the H\"ormander-type condition. In Definition \ref{def_Hslu} and subsequent 
discussion we introduce and investigate uniformly local right- and left- Sobolev spaces
$H^s_{l.u.,R}(\Gh)$ and $H^s_{l.u.,L}(\Gh)$ on the unitary dual $\Gh$, respectively.
Using these spaces we can then define uniformly local Sobolev spaces on $\Gh$ by
$$ H^s_{l.u.}(\Gh):= H^s_{l.u.,R}(\Gh) \bigcap H^s_{l.u.,L}(\Gh),$$
with the norm
$$
\|\sigma\|_{H^s_{l.u.},\eta,\cR}:=\max\left(\|\sigma\|_{H^s_{l.u.,R},\eta,\cR},\|\sigma\|_{H^s_{l.u.,L},\eta,\cR}\right),
$$
depending on a choice of $\eta \in \cD(0,\infty)$ and a
positive Rockland operator $\cR$, and in Proposition \ref{prop_Hslu_norm} we
show that different choices of $\eta$ and $\cR$ lead to
equivalent norms.
Then we have

\begin{theorem} 
\label{thm_main0}
Let $G$ be a graded nilpotent Lie group.
Let  $\sigma=\{\sigma(\pi),\pi\in \Gh\}$ be
a measurable field of operators in $L^2(\Gh)$.
If $\sigma\in H^s_{l.u.}(\Gh)$
for some $s>Q/2$, where $Q$ is the homogeneous dimension of $G$,
then the corresponding operator $T=T_\sigma$ is bounded on $L^p(G)$ for any $1<p<\infty$.
Furthermore, 
$$
\|T\|_{\sL(L^p(G))}
\leq C 
\|\sigma\|_{H^s_{l.u.}\eta,\cR},
$$
where $C>0$ is a constant independent of $\sigma$
but may depend on $p,s,G$ and a choice of $\eta \in \cD(0,\infty)$ and a
positive Rockland operator $\cR$.
\end{theorem}
Theorem \ref{thm_main0} will be reformulated in Theorem \ref{thm_main} and its statement will be refined 
further in Corollary \ref{cor_thm_main}.

\bigskip

The paper is organised as follows. 
In Section \ref{sec_preliminary}, 
we recall the necessary material regarding the setting.
In Section \ref{sec_HsGh}, 
we define and study Sobolev spaces on $\Gh$.
In Section \ref{sec_MHcond}, 
we present our Mihlin-H\"ormander condition.
In Section \ref{sec_proofs}, 
we show the statements of the previous section.

\bigskip

\noindent\textbf{Notation:}
$\bN_0=\{0,1,2,\ldots\}$ denotes the set of non-negative integers
and 
$\bN=\{1,2,\ldots\}$ the set of positive integers.
If $\cH_1$ and $\cH_2$ are two Hilbert spaces, we denote by $\sL(\cH_1,\cH_2)$ the Banach space of the bounded operators from $\cH_1$ to $\cH_2$. If $\cH_1=\cH_2=\cH$ then we write $\sL(\cH_1,\cH_2)=\sL(\cH)$.
We may allow ourselves to write 
$A\lesssim B$ when $A$ is less than $B$ up to a constant, 
and $A\asymp B$ when the quantity $A$ and $B$ are equivalent in the sense that there exists a constant such that $C^{-1}A\leq B\leq C A$.

\section{Preliminaries}
\label{sec_preliminary}

In this section, after defining graded Lie groups, 
we recall their homogeneous structure, 
some general representation theory in this context  as well as the definition and some properties of
their Rockland operators.

\subsection{Graded and homogeneous Lie groups}

Here we recall briefly the definition of graded nilpotent Lie groups
and their natural homogeneous structure.
A complete description of the notions of graded and homogeneous nilpotent Lie groups may be found in \cite[ch1]{FS}
and \cite[ch3]{FR-monograph}.

We will be concerned with graded Lie groups $G$
which means that $G$ is a connected and simply connected 
Lie group 
whose Lie algebra $\mathfrak g$ 
admits an $\bN$-gradation
$\mathfrak g= \oplus_{\ell=1}^\infty \mathfrak g_{\ell}$
where the $\mathfrak g_{\ell}$, $\ell=1,2,\ldots$, 
are vector subspaces of $\mathfrak g$,
almost all equal to $\{0\}$,
and satisfying 
$[\mathfrak g_{\ell},\mathfrak g_{\ell'}]\subset\mathfrak g_{\ell+\ell'}$
for any $\ell,\ell'\in \bN$.
This implies that the group $G$ is nilpotent.
Examples of such groups are the Heisenberg group and, more generally,
all stratified Lie groups (which by definition correspond to the case $\fg_1$ generating the full Lie algebra $\fg$).

We construct a basis $X_1,\ldots, X_n$  of $\fg$ adapted to the gradation,
by choosing a basis $\{X_1,\ldots X_{n_1}\}$ of $\mathfrak g_1$ (this basis is possibly reduced to $\emptyset$), then 
$\{X_{n_1+1},\ldots,  X_{n_1+n_2}\}$ a basis of $\mathfrak g_2$
(possibly $\emptyset$ as well as the others)
and so on.
Via the exponential mapping $\exp_G : \mathfrak g \to G$, we   identify 
the points $(x_{1},\ldots,x_n)\in \bR^n$ 
 with the points  $x=\exp_G(x_{1}X_1+\cdots+x_n X_n)$ in $G$.
Consequently we allow ourselves to denote by $C(G)$, $\cD(G)$ and $\cS(G)$ etc,
the spaces of continuous functions, of smooth and compactly supported functions or 
of Schwartz functions on $G$ identified with $\bR^n$,
and similarly for distributions with the duality notation 
$\langle \cdot,\cdot\rangle$.

This basis also leads to a corresponding Lebesgue measure on $\mathfrak g$ and the Haar measure $dx$ on the group $G$,
hence $L^p(G)\cong L^p(\bR^n)$.
The group convolution of two functions $f$ and $g$, 
for instance   integrable, 
is defined via 
$$
 (f*g)(x):=\int_\G f(y) g(y^{-1}x) dy.
$$
The convolution is not commutative: in general, $f*g\not=g*f$, 
but the Young convolutions inequalities hold, so that we have
\begin{equation}
\label{eq_young}
\|f*g\|_{L^r(G)}\leq \|f\|_{L^p(G)}\|g\|_{L^q(G)}
,\quad p,q,r\in [1,\infty], \
1+\frac 1r =\frac 1p +\frac 1q.
\end{equation}

The coordinate function $x=(x_1,\ldots,x_n)\in G\mapsto x_j \in \bR$
is denoted by $x_j$.
More generally we define for every multi-index $\alpha\in \bN_0^n$,
$x^\alpha:=x_1^{\alpha_1} x_2 ^{\alpha_2}\ldots x_{n}^{\alpha_n}$, 
as a function on $G$.
Similarly we set
$X^{\alpha}=X_1^{\alpha_1}X_2^{\alpha_2}\cdots
X_{n}^{\alpha_n}$ in the universal enveloping Lie algebra $\fU(\fg)$ of $\mathfrak g$.

For any $r>0$, 
we define the  linear mapping $D_r:\mathfrak g\to \mathfrak g$ by
$D_r X=r^\ell X$ for every $X\in \mathfrak g_\ell$, $\ell\in \bN$.
Then  the Lie algebra $\mathfrak g$ is equipped 
with the family of dilations  $\{D_r, r>0\}$
and becomes a homogeneous Lie algebra in the sense of 
\cite{FS}.
We re-write the set of integers $\ell\in \bN$ such that $\fg_\ell\not=\{0\}$
into the increasing sequence of positive integers
 $\upsilon_1,\ldots,\upsilon_n$ counted with multiplicity,
 the multiplicity of $\fg_\ell$ being its dimension.
 In this way, the integers $\upsilon_1,\ldots, \upsilon_n$ become 
 the weights of the dilations and we have $D_r X_j =r^{\upsilon_j} X_j$, $j=1,\ldots, n$,
 on the chosen basis of $\fg$, and we have $X_j \in \fg_{\upsilon_j}$ for $j=1,\ldots, n$.
 The associated group dilations are defined by
$$
D_r(x)=
r\cdot x
:=(r^{\upsilon_1} x_{1},r^{\upsilon_2}x_{2},\ldots,r^{\upsilon_n}x_{n}),
\quad x=(x_{1},\ldots,x_n)\in G, \ r>0.
$$
In a canonical way  this leads to the notions of homogeneity for functions and operators.
For instance
the degree of homogeneity of $x^\alpha$ and $X^\alpha$,
viewed respectively as a function and a differential operator on $G$, is 
$$
[\alpha]=\sum_j \upsilon_j\alpha_{j}.
$$
Indeed, let us recall 
that a vector of $\mathfrak g$ defines a left-invariant vector field on $G$ 
and, more generally, 
that the universal enveloping Lie algebra of $\mathfrak g$ 
is isomorphic with the left-invariant differential operators; 
we keep the same notation for the vectors and the corresponding operators. 

Recall that a \emph{homogeneous quasi-norm} on $G$ is a continuous function $|\cdot| : G\rightarrow [0,+\infty)$ homogeneous of degree 1
on $G$ which vanishes only at 0. This often replaces the Euclidean norm in the analysis on homogeneous Lie groups, for instance in the following well-known properties:
\begin{proposition}
\label{prop_homogeneous_quasi_norm}
\begin{enumerate}
\item 
Any homogeneous quasi-norm $|\cdot|$ on $G$ satisfies a triangle inequality up to a constant:
$$
\exists C\geq 1 \quad \forall x,y\in G\quad
|xy|\leq C (|x|+|y|).
$$
It partially satisfies the reverse triangle inequality:
\begin{equation}
\label{eq_reverse_triangle}
\forall b\in (0,1) \quad
\exists C=C_b\geq 1 \quad \forall x,y\in G
\quad |y|\leq b |x| \Longrightarrow
\big| |xy| - |x|\big|\leq C|y|
.
\end{equation}
\item
Any two homogeneous quasi-norms $|\cdot|_1$ and $|\cdot|_2$ are equivalent in the sense that
$$
\exists C>0 \quad \forall x\in G\quad
C^{-1} |x|_2\leq |x|_1\leq C |x|_2.
$$
\end{enumerate}
\end{proposition}

 An example of a homogeneous quasi-norm is given via
\begin{equation}
\label{eq_quasinormnuo}
|x|_{\nu_o}:=\Big(\sum_{j=1}^n x_j^{2\nu_o/\upsilon_j}\Big)^{1/{2\nu_o}},
\end{equation}
with $\nu_o$ a common multiple to the weights $\upsilon_1,\ldots,\upsilon_n$.
 
We will use the Young inequalities together with the properties of quasi-norms in the following way:
\begin{lemma}
\label{lem_L1omegas}
Let $|\cdot|$ be a quasi-norm and let $s\geq 0$.
We set $\omega_s=(1+|\cdot|)^s$.
Let $p,q,r$ be as in Young's inequality in \eqref{eq_young}.
Then if $f$ and $g$ are measurable functions, then 
the following inequality holds (with possibly unbounded quantities):
$$
\| \omega_s \ f*g\|_{L^r(G)}
\leq 
C
\| \omega_s \ f\|_{L^p(G)}
\| \omega_s \ g\|_{L^q(G)},
$$
where the constant $C$ is independent on $f,g$
but may depend on $s,G, |\cdot|$.
\end{lemma}
 
 \begin{proof}
The triangular inequality (see Proposition \ref{prop_homogeneous_quasi_norm}) implies easily
\begin{equation}
\label{eq_omegasxy}
\exists C=C_{s,|\cdot|}\quad\forall x,y\in G\quad
\omega_s(x)\leq 
C 
\omega_s(xy^{-1})
\omega_s(y),
\end{equation}
yielding 
$\omega_s(x) |f*g|(x) 
\leq
C |\omega_s f| * |\omega_s g|.$
We conclude with Young's inequality
(see \eqref{eq_young}).
\end{proof}

 Various aspects of analysis on $G$ can be developed in a comparable way with the Euclidean setting sometimes replacing the topological dimension 
$n =\sum_{\ell=1}^\infty\dim \fg_\ell$
of the group $G$ by its \emph{homogeneous dimension}
$$
Q:=\sum_{\ell=1}^\infty \ell \dim \fg_\ell
=  \upsilon_1 +\upsilon_2 +\ldots +\upsilon_n 
 .
$$

For example, there is an analogue of polar coordinates on homogeneous 
groups with $Q$ replacing $n$, see \cite{FS}:
\begin{equation}
\label{formula_polar_coord}
\forall f\in L^1(G)
\qquad \int_G f(x)dx
=
\int_0^\infty \int_\fS f(ry) r^{Q-1} d\sigma(y) dr,
\end{equation}
where $\sigma$ is a (unique) positive Borel measure on 
the unit sphere 
$\fS:=\{x\in G\, : \, |x|=1\}$.
This implies the following simple embeddings:

\begin{corollary}
\label{cor_prop_polar_coord}
Let $|\cdot|$ be a fixed homogeneous quasi-norm on $G$.
If $s> Q/2$, 
then there exists $C>0$ such that for any measurable function $f$ we have
$$
\|f\|_{L^1(G)} \leq C \|(1+|\cdot|)^s f\|_{L^2(G)}
$$
Moreover as long as $s-\epsilon>Q/2$, 
 there exists $C>0$ such that for any measurable function $f$ we have
$$
\|(1+|\cdot|)^\epsilon f\|_{L^1(G)} \leq C \|(1+|\cdot|)^s f\|_{L^2(G)}
$$
\end{corollary}

\begin{proof}[Proof of Corollary \ref{cor_prop_polar_coord}]
By Cauchy-Schwartz' or H\"older's inequality, 
we have
$$
\|(1+|\cdot|)^\epsilon f\|_{L^1(G)} \leq C_{s,\epsilon} \|(1+|\cdot|)^s f\|_{L^2(G)},
$$
where $C_s:=\|(1+|\cdot|)^{-s +\epsilon}\|_{L^2(G)}$.
Using the polar change of coordinates \eqref{formula_polar_coord}, $C_s$
 is finite for $s -\epsilon>Q/2$.
\end{proof}

We will need an $L^1$-mean value property:
\begin{lemma}
\label{lem_L1mv}
There exists $C>0$ such that 
for any $h\in G$ and any $f\in C^1(G)$
we have
$$
\|f-f(\cdot\, h)\|_{L^1(G)} \leq C \sum_{\ell=1}^n |h|^{\upsilon_\ell}
\|X_\ell f\|_{L^1(G)},
$$
and 
$$
\|f-f(h\,\cdot )\|_{L^1(G)} \leq C \sum_{\ell=1}^n |h|^{\upsilon_\ell}
\|\tilde X_\ell f\|_{L^1(G)}.
$$
\end{lemma}

In this paper, if $X\in \fg$, 
then we keep the same notation $X$ for the left invariant vector field while
 $\tilde X$ denotes the right invariant vector field, 
 that is, we have for any function $f\in C^\infty(G)$ and $x\in G$:
 $$
 Xf(x) = \frac{d}{ds}|_{s=0} f\left(x\exp_G(sX)\right)
 \quad\mbox{while}\quad 
 \tilde Xf(x) = \frac{d}{ds}|_{s=0} f\left(\exp_G(sX)x\right).
 $$
We adapt the argument of \cite[Mean Value Theorem 1.33]{FS}
and \cite[\S 3.1.8]{FR-monograph}.

\begin{proof}[Proof of Lemma \ref{lem_L1mv}]
Any $h\in G$ may be written as 
$$
h=h_1\ldots h_n
\quad\mbox{with}\quad 
h_\ell:= \exp (t_\ell X_\ell)
\quad\mbox{and}\quad 
|t_\ell|\leq C |h|^{1/\upsilon_\ell}.
$$
Therefore, we have
\begin{eqnarray*}
\|f-f( h\, \cdot)\|_{L^1(G)}
&\leq& 
\sum_{j=1}^n 
\int_G |f(h_j h_{j+1} \ldots h_n x)
-f(h_{j+1} \ldots h_n x)|dx
\\
&\leq& 
\sum_{j=1}^n 
\int_{G\times [0,t_j]}  
|\tilde X_j f(\exp (s X_j) h_{j+1} \ldots h_n x)|dx ds
\\
&&= 
\sum_{j=1}^n 
\int_{G\times [0,t_j]} 
 |\tilde X_j f(y)|dy ds
=
\sum_{j=1}^n 
|t_j| \int_G |\tilde X_j f(y)|dy. 
\end{eqnarray*}

This shows the right case and 
the left case is similar.
\end{proof}

\subsection{The dual of $G$ and the Plancherel theorem}
\label{subsec_Gh+plancherel}

Here we set some notation and recall some properties 
regarding the representations of the group $G$, 
especially the Plancherel theorem,
and its enveloping Lie algebra $\fU(\fg)$. 
The (very) general theory may be found in \cite{dixmier}, 
for a description more adapted to our particular context, 
see \cite[ch1]{FR-monograph}.
Note that we will not use the orbit method \cite{corwingreenleaf}.

In this paper,  we always assume that the representations of the group $G$ 
are strongly continuous and acting on separable Hilbert spaces.
For a unitary representation $\pi$ of $G$, 
we keep the same notation for the corresponding infinitesimal representation
which acts on the universal enveloping algebra $\fU(\fg)$ of the Lie algebra of the group.
It is characterised by its action on $\fg$:
\begin{equation}
\label{eq_def_pi(X)}
\pi(X)=\partial_{t=0}\pi(e^{tX}),
\quad  X\in \fg.
\end{equation}
The infinitesimal action acts on the space $\cH_\pi^\infty$
of smooth vectors, that is, the space of vectors $v\in \cH_\pi$ such that 
 the function $G\ni x\mapsto \pi(x)v\in \cH_\pi$ is of class $C^\infty$.

For a unitary representation $\pi$ and any $f\in L^1(G)$,
we define the operator
$$
\pi(f)
=\int_G f(x) \pi(x)^*dx.
$$ 
One checks easily
\begin{equation}
\label{eq_Fnorm_L1}
\|\pi(f) \|_{\sL(\cH_\pi)} \leq \|f\|_{L^1(G)}.
\end{equation}

We denote by $\Gh$ the set of classes of unitary irreducible representations
modulo unitary equivalence, see \cite{dixmier} or \cite{FR-monograph}.
It is a standard Borel space (i.e. a separable complete metrisable topological space equipped with the sigma-algebra generated by the open sets).

From now on, we may identify an unitary irreducible representation 
with its class in $\Gh$.
This leads to the notion of  \emph{group Fourier transform} for a function  $f\in L^1(G)$ at $\pi\in \Gh$.
$$
 \pi(f) \equiv \widehat f(\pi) \equiv \cF_G(f)(\pi).
$$

The Plancherel measure is the unique positive Borel standard measure $\mu$ on $\Gh$ such that 
for any $f\in C_c(G)$, we have
$$
\int_G |f(x)|^2 dx = \int_{\Gh} \|\cF_G(f)(\pi)\|_{HS(\cH_\pi)}^2 d\mu(\pi).
$$
Here $\|\cdot\|_{HS(\cH_\pi)}$ denotes the Hilbert-Schmidt norm on the space $HS(\cH_\pi) \sim \cH_\pi \otimes \cH_\pi^*$ of Hilbert-Schmidt operators on the Hilbert space  $\cH_\pi$.
This implies that the group Fourier transform extends unitarily from 
$L^1(G)\cap L^2(G)$ to $L^2(G)$ onto 
$$
L^2(\Gh):=\int_{\Gh} \cH_\pi \otimes\cH_\pi^* d\mu(\pi),
$$
which we identify with the space of $\mu$-square integrable fields on $\Gh$. The Plancherel formula may be rephrased as
\begin{equation}
\label{eq_plancherel_formula}
\|f\|_{L^2(G)} 
=
\|\widehat f\|_{L^2(\Gh)}. 
\end{equation}

The orbit method furnishes an expression for the Plancherel measure $\mu$,
see \cite[Section 4.3]{corwingreenleaf}.
However we will not need this here.

The general theory on locally compact unimodular group of type I applies 
\cite{dixmier}:
let $\sL(L^2(G))$ be the space of bounded linear operators on $L^2(G)$
and 
let  $ \sL_L(L^2(G))$ be the subspace of those operators $T\in \sL(L^2(G))$ which are left-invariant, that is, commute with the left translation:
$$
T(f(g\,\cdot))(g_1) =(Tf)(gg_1), \quad f\in L^2(G), \ g,g_1\in G.
$$ 
Then there exists a field of bounded operators $\widehat T (\pi)\in \sL(\cH_\pi)$, $\pi\in \Gh$, such that
$$
\forall f\in L^2(G)\quad
\cF_G(Tf)(\pi) = \widehat T(\pi) \ \widehat f(\pi)
\quad \mbox{for} \ \mu-\mbox{almost all}\ \pi\in \Gh.
$$
Moreover the operator norm of $T$ is equal to 
$$
\|T\|_{\sL(L^2(G))}=\sup_{\pi\in \Gh} \|\widehat T(\pi)\|_{\sL(\cH_\pi)}.
$$
The supremum here has to be understood as the essential supremum with respect to the Plancherel measure $\mu$.
By the Schwartz kernel theorem, 
any operator $T\in \sL_L(L^2(G))$ is a convolution operator 
and we denote by $T\delta_0\in \cS'(G)$ its convolution kernel:
$Tf = f * (T\delta_0)$, $f\in \cS(G)$.
One may extend the definition of the group Fourier transform to these distributions via 
$\cF_G\{T\delta_0\}=\widehat T(\pi)$.

Denoting by $L^\infty(\Gh)$ 
the space of fields of  operators $\sigma_\pi\in \sL(\cH_\pi)$, $\pi\in \Gh$, with 
$$
\|\sigma\|_{L^\infty(\Gh)}:=
\sup_{\pi\in \Gh} \|\sigma_\pi\|_{\sL(\cH_\pi)}<\infty,
$$
modulo equivalence under the Plancherel measure $\mu$, 
we have obtained that 
$T\in \sL(L^2(G))$ implies 
$\{\cF_G\{T\delta_0\}=\widehat T(\pi),\pi\in \Gh\}\in  L^\infty(\Gh)$.
Conversely, 
to any  field $\sigma=\{\sigma_\pi,\pi\in \Gh\}$ in $L^\infty(\Gh)$, 
we associate the \emph{Fourier multiplier operator} $T_\sigma$
via
\begin{equation}
\label{eq_F_mult}
\cF_G\{T_\sigma(\phi)\}(\pi)=\sigma_\pi \widehat \phi(\pi),
\quad \phi\in L^2(G).
\end{equation}
The Plancherel formula implies that $T_\sigma\in \sL_L(L^2(G))$
with operator norm bounded by $\|\sigma\|_{L^\infty(\Gh)}$.
As recalled above, the operator norm is in fact equal to the $L^\infty(\Gh)$-norm of $\sigma$.
Thus we have obtained the isometric isomorphism of von Neumann algebras
$$
\left\{\begin{array}{rcl}
L^\infty(\Gh) &\longrightarrow&  \sL_L(L^2(G))\\
\sigma &\longmapsto& T_\sigma
\end{array}\right.
$$
with inverse given via $\sigma= \cF_G \{T_\sigma\delta_0\}$.

\subsection{Rockland operators}
\label{subsec_rockland}

Here we recall the definition of Rockland operators
and their main properties.

\begin{definition}
A \emph{Rockland operator}
\index{Rockland!operator}
 $\cR$ on $G$ is 
a left-invariant differential operator on $G$
which is homogeneous of positive degree and
such that  
for each unitary irreducible non-trivial representation $\pi$ on $G$,
the operator $\pi(\cR)$ is injective on $\cH_\pi^\infty$, 
that is,
$$
 \forall v\in \cH_\pi^\infty\qquad
\pi(\cR)v = 0 \ \Longrightarrow \ v=0
.
$$
\end{definition}

Although the definition of a Rockland operator would make sense on a homogeneous Lie group (in the sense of  \cite{FS}), 
it turns out 
that the existence of a (differential) Rockland operator on a homogeneous group 
implies that the homogeneous group may be assumed to be graded
(cf.  \cite{miller,TElst+Robinson}, see also 
\cite[Proposition 4.1.3]{FR-monograph}).
This explains why we have chosen the setting of graded Lie groups for this paper.
Helffer and Nourrigat proved
\cite{helffer+nourrigat-79} the Rockland conjecture, 
that is, that Rockland operators are all the hypoelliptic 
left-invariant differential operators on a given graded Lie group.
Hence Rockland operators may be viewed as 
analogues of elliptic operators or more generally hypoelliptic operator (with any degree of homogeneity)
in a non-abelian  context.

Some authors may have different conventions than ours regarding Rockland operators: for instance some choose to consider right-invariant operators
and some consider operators which are not necessarily homogeneous. 
However, the choice of conventions does not interfere with the study of the objects in themselves.

\begin{ex}
In the stratified case, one can check easily that
any (left-invariant negative)  \emph{sub-Laplacian}, 
that is
\begin{equation}
\label{eq_def_lih_subLapl}
\cL=Z_1^2+\ldots+Z_{n'}^2
\quad\mbox{with} \ Z_1,\ldots,Z_{n'} \
\mbox{forming any basis of the first stratum} \ \fg_1,
\end{equation}
 is a Rockland operator.
\end{ex}

\begin{ex}
On any graded group $G$, it is not difficult to see that 
 the operator
\begin{equation}
\label{eq_cR_ex}
\sum_{j=1}^n
(-1)^{\frac{\nu_o}{\upsilon_j}}
 c_j X_j^{2\frac{\nu_o}{\upsilon_j}}
\quad\mbox{with}\quad c_j>0
,
\end{equation}
is a Rockland operator of homogeneous degree $2\nu_o$
if $\nu_o$ is any common multiple of $\upsilon_1,\ldots, \upsilon_n$.
\end{ex}

Hence Rockland operators do exist on any graded Lie group 
(not necessarily stratified).

If the Rockland operator $\cR$  is formally self-adjoint,
that is, $\cR^*=\cR$
as elements of the universal enveloping algebra $\fU(\fg)$,
then $\cR$ and $\pi(\cR)$ admit self-adjoint extensions
on $L^2(G)$ and $\cH_\pi$ respectively, see [ch.3.B]\cite{FS} or 
\cite[\S 4.1.3]{FR-monograph}.
We keep the same notation for their self-adjoint extensions.
We denote by    $E$ and $E_\pi$ their spectral measure:
$$
\cR = \int_\bR \lambda dE (\lambda)
\quad\mbox{and}\quad
\pi(\cR) = \int_\bR \lambda dE_\pi(\lambda).
$$

We will be interested in the positive Rockland operators:
$$
\forall f \in \cS(G)\qquad
\int_G \cR f(x) \ \overline{f(x)} dx\geq 0.
$$
They are formally self-adjoint.
One checks easily that the operator  in \eqref{eq_cR_ex} is positive.
This shows that positive Rockland operators always exist 
on any graded Lie group.
Note that if $G$ is stratified and $\cL$ is a (left-invariant negative) sub-Laplacian as in \eqref{eq_def_lih_subLapl}, 
then it is customary to privilege   $-\cL$ as a positive Rockland operator.

The point 0 in the spectrum of a positive Rockland operator is negligible with respect to the spectral measure, 
see \cite{hula+ludwig} or \cite[Remark 4.2.8.4]{FR-monograph}.
Consequently, one can define multipliers operators of $\cR$ on $(0,+\infty)$, 
the value of this multiplier function at 0 being negligible. 
The properties of the functional calculus of $\cR$ and of the group 
Fourier transform imply

\begin{lemma}
\label{lem_homogeneity_multipliers}
Let $\cR$ be a positive Rockland operator of homogeneous degree $\nu$
and  $f:\bR^+\to \bC$ be a measurable function.
We assume that the domain of the operator $ f(\cR)= \int_\bR f(\lambda) dE (\lambda) $ 
contains $\cS(G)$.
Then for any $\phi\in \cS(G)$, 
$$
 \left( f(r^\nu \cR) \phi\right)\circ D_r =f(\cR) \left(\phi \circ D_r\right),
$$
where $\nu$ denotes the homogeneous degree of $\cR$,
and in the sense of distribution
\begin{equation}
\label{eq_homogeneity_kernel_multipliers}
 f(r^\nu \cR) \delta_0 (x) =r^{-Q} f(\cR)\delta_0(r^{-1}x),
 \quad x\in G, 
\end{equation}
where  $f(\cR)\delta_0$ denotes the right convolution kernel of $f(\cR)$.
\end{lemma}

Let us recall Hulanicki's Theorem, see \cite{hula} or \cite[\S4.5]{FR-monograph}.
\begin{theorem}[Hulanicki]
\label{thm_hula}
Let $|\cdot|$ be a quasi-norm on $G$, $s\geq0$, $p\in [1,\infty)$,
 and $\alpha\in \bN_0^n$.
Then there exists $C>0$ and $d\in \bN$
such that
for any  $f\in C^d(0,\infty)$,
we have
$$
\int_G (1+|x|)^s |X^\alpha f(\cR)\delta_0(x)|^p dx \leq C
\sup_{\lambda>0, \ell=0,\ldots d}
(1+\lambda)^d |f^{(\ell)}(\lambda)|,
$$
 provided that the supremum on the right-hand side is finite.
 
 The same results with the right-invariant vector fields  $\tilde X_j$'s instead of the left-invariant vector fields $X_j$'s hold.
 
Consequently, if $f$ is a Schwartz function, 
that is, $f\in \cS(\bR)$
 (for instance in $\cD(\bR)$), then $f(\cR)\delta_0\in \cS(G)$. 
\end{theorem}

We will also use the fact that any two positive Rockland operators are equivalent in the following sense 
(see \cite{FR-sob} or \cite[\S4.4.5, especially Corollary 4.4.21]{FR-monograph}):
\begin{proposition}
\label{prop_powercR1cR2}
\begin{itemize}
\item If $\cR$ is a positive Rockland operator, 
then for any $s\geq 0$ the powers $\cR^s$ 
defined by spectral calculus
are (unbounded) operators on $L^2(G)$ with domains containing $\cS(G)$.
\item 
Let $\cR_1$ and $\cR_2$ be two positive Rockland operators 
of homogeneous degrees $\nu_1$ and $\nu_2$ respectively.
Then for any $s \geq 0 $ we have
$$
\exists C>0\qquad
\forall \phi\in \cS(G)\qquad
\|\cR_1^{s/\nu_1}  \phi \|_{L^2(G)}
\leq C\|\cR_2^{s/\nu_2}  \phi \|_{L^2(G)}.
$$
\end{itemize}

\end{proposition}

Note that Proposition \ref{prop_powercR1cR2} implies that 
if the hypothesis in \eqref{eq_hyp_thm_intro} 
of Theorem \ref{thm_intro}
is satisfied for one positive Rockland operator 
then it is satisfied for all.

\subsection{Difference operators}
The difference operators are aimed at replacing the derivatives with respect to the Fourier variable in the Euclidean case.

If $q$ is a continuous function on $G$, 
we define $\Delta_q$ via
$$
\Delta_q \widehat f (\pi) = \cF_G(q f) (\pi),
\quad \pi\in \Gh,
$$
for any function $f\in \cD(G)$.
As the group Fourier transform is injective
and since $\cD(G)$ is dense in $L^p(G)$, $p\in [1,\infty)$,
this defines the difference operator $\Delta_q$
as a (possibly) unbounded operator 
with domains in $L^2(\Gh)$ or $\cF L^1(G)$ and values in $L^\infty(\Gh)$.
In particular, for  $\alpha\in \bN_0^n$, 
we set
$$
\Delta^\alpha:=\Delta_{x^\alpha}.
$$

\begin{remark}
\label{rem_def_Delta}
Assuming $q$ to be a continuous function with polynomial growth, one can define difference operators on $\cS(G)$.
Moreover, 
assuming further hypotheses on $q$, 
difference operators may be defined on the image of the group Fourier transform of more general distributional spaces on $G$
where $\cD(G)$ is not necessarily dense for instance on $\cF_G^{-1} L^\infty(\Gh)$.
In fact, in \cite[Section 5.2.1]{FR-monograph}, difference operators are defined in a slightly more general context.
\end{remark}

The difference operators as defined above  were described concretely in the case of the Heisenberg groups
\cite[Section 6.3]{FR-monograph}, and 
one checks easily that they coincide with the difference-differentials operators of \cite{dM+M}, 
see also \cite{geller1980} and \cite{bahouri+chemin+danchin2018,bahouri+chemin+danchin2019}.
In the case of compact Lie groups, an intrinsic notion of difference operators can be defined even on symbols that are not Fourier transforms of distributions, see \cite{fischerJFA2015,fischerJFA2020}.
On a general Lie group (even restricting oneself to the nilpotent class), 
to the authors' knowledge at the time of writing, there is not a more intrinsic way to define difference operators than the one above.

\smallskip

The difference operators satisfy the Leibniz rule
 \cite[\S5.2.2]{FR-monograph}:
\begin{equation}
\label{eq_Leibniz_rule_sigma}
\Delta^\alpha (\sigma_1\sigma_2)=
\sum_{[\alpha_1]+[\alpha_2]= [\alpha]} 
c_{\alpha_1,\alpha_2}\Delta^{\alpha_1}\sigma_1 \ 
\Delta^{\alpha_2}\sigma_2,
\end{equation}
where $c_{\alpha_1,\alpha_2}$ are universal constants.
By `universal constants', 
we mean that they depend only on $G$ and the choice of the basis $\{X_j\}_{j=1}^n$.

\section{The Sobolev spaces on $\Gh$}
\label{sec_HsGh}

The aim of this section is to study Sobolev-type spaces on $\Gh$
defined in the following way:

\begin{definition}
\label{def_HshatG}
For each $s \geq0$,
we define $H^s(\Gh)$ 
as the space of measurable fields $\sigma=\{\sigma(\pi)\}$
such that $\sigma\in L^2(\Gh)$ and $\Delta_{(1+|\cdot|)^{s}}\sigma \in L^2(\Gh)$ 
where $|\cdot|$ is a quasi-norm on $G$.
 \end{definition}

This means that $H^s(\Gh)$ is the image via the group Fourier transform of the subspace 
$L^2(G, (1+|\cdot|)^{2s})$ of $L^2(G)$:
$$
H^s(\Gh) = \cF_G \left(L^2(G, (1+|\cdot|)^{2s})\right).
$$
We will call the spaces $H^s(\Gh)$ the Sobolev spaces on $\Gh$.
This vocabulary is justified by the properties stated and proved in this section.
We start by showing that the Sobolev spaces on $\Gh$ are Hilbert spaces
independent of the quasi-norm:
\begin{proposition}
\label{prop_HsGh_hilbert}
Let $s \geq0$.
\begin{enumerate}
\item
\label{item_prop_HsGh_hilbert_indep_norm}
 The space $H^s(\Gh)$ is independent of the quasi-norm $|\cdot|$.
\item 
\label{item_prop_HsGh_hilbert_equiv}
Let  $\sigma=\{\sigma(\pi),\pi\in \Gh\}$ be
a $\mu$-measurable field of operators and let $s\geq 0$.
The following conditions are equivalent:
\begin{itemize}
\item $\sigma\in H^s(\Gh)$,
\item there exists a quasi-norm $|\cdot|'$ such that 
$\cF_G^{-1} \sigma\in L^2(G, (1+|\cdot|')^{2s})$,
\item 
$\cF_G^{-1} \sigma\in L^2(G, (1+|\cdot|')^{2s})$
for any quasi-norm $|\cdot|'$,
\item 
$\cF_G^{-1} \sigma\in L^2(G,\omega_{s}^2)$
for any continuous function $\omega_s:G\to (0,\infty)$ 
equivalent to $(1+|\cdot|)^s$ in the sense that
\begin{equation}
\label{eq_omega_s}
\exists C>0 \quad \forall x\in G\qquad
 C^{-1} (1+|x|)^s
\leq \omega_s(x)\leq C (1+|x|)^s,
\end{equation}
for one (and then all) quasi-norm $|\cdot|$.
\end{itemize}
\item
\label{item_prop_HsGh_hilbert_Hilbertspace}
Fixing a weight $\omega_s$ satisfying \eqref{eq_omega_s},
the space $H^s(\Gh)$ is a Hilbert space when equipped 
with the sesquilinear form given via
\begin{eqnarray*}
(\sigma_1,\sigma_2)_{H^s} 
&:=&
(\Delta_{\omega_s} \sigma_1,\Delta_{\omega_s} \sigma_2)_{L^2(\Gh)} 
\\
&=&
\int_{\Gh} \tr\left( \Delta_{\omega_s} \sigma_1(\pi) \  
\Delta_{\omega_s} \sigma_2(\pi)^*\right)
d\mu(\pi).
\end{eqnarray*}
The corresponding norm is given by
$$
\|\sigma\|_{H^s,\omega_s}:=
\|\Delta_{\omega_s} \sigma\|_{L^2(\Gh)}
=
\| \omega_s \cF_G^{-1} \sigma\|_{L^2(G)}.
$$
Any two weights $\omega^{(1)}_s$ and $\omega^{(2)}_s$
 satisfying \eqref{eq_omega_s}
yield equivalent norms on $H^s(\Gh)$.
\end{enumerate}
\end{proposition}

 \begin{proof}[Proof of Proposition \ref{prop_HsGh_hilbert}]
For any $\omega_s$  satisfying \eqref{eq_omega_s},
 and any quasi-norm $|\cdot|$,
we have
$L^2(G, (1+|\cdot|)^s) = L^2(G, \omega_s)$.
If $|\cdot|'$ is another quasi-norm,
$ (1+|\cdot|')^s)$ is a continuous function satisfying \eqref{eq_omega_s}
since  two quasi-norms are equivalent
by Proposition \ref{prop_homogeneous_quasi_norm}.
This together with the isometry $\cF_G:L^2(G)\to L^2(\Gh)$ 
between Hilbert spaces
imply the statement.
\end{proof}

We may allow ourselves to denote the $H^s(\Gh)$-norm by 
$$
\|\sigma\|_{H^s}:=
\|\sigma\|_{H^s,\omega_s},
$$
when a function $\omega_s$ has been fixed.

The space $H^s(\Gh)$ is stable by taking the adjoint
as one checks easily the following property: 
if $\sigma=\{\sigma(\pi),\pi\in \Gh\}$ is in $H^s(\Gh)$ 
then 
$\sigma^*=\{\sigma(\pi)^*,\pi\in \Gh\}$ is also in $H^s(\Gh)$
and 
\begin{equation}
\label{eq_sigma*Hs}
\|\sigma\|_{H^s, (1+|\cdot|)^s} =
\|\sigma^*\|_{H^s,(1+|\cdot|)^s}.
\end{equation}

We have the following  inclusions and log-convexity.
\begin{lemma}
\label{lem_HsGh_inclusion}
The following continuous inclusions holds
for $s_2\geq s_1\geq 0$.
$$
L^2(\Gh)=H^0(\Gh) \supset H^{s_1}(\Gh) \supset H^{s_2}(\Gh).
$$

If $s$ is between the two non-negative numbers $s_1$ and $s_2$,
then  
$$
\|\sigma\|_{H^s,\omega_s}
\leq
\|\sigma\|_{H^{s_1},\omega_{s_1}}^\theta
\|\sigma\|_{H^{s_2},\omega_{s_2}}^{1-\theta},
$$
having written $s= \theta s_1 +(1-\theta)s_2$,  
with  $\theta\in [0,1]$,
and fixed a quasi-norm $|\cdot|$ and $\omega_s=(1+|\cdot|)^s$. 
\end{lemma}
\begin{proof}
The inclusions follow readily from 
$(1+|\cdot|)^{s_1}\leq (1+|\cdot|)^{s_2}$ when  $s_2\geq s_1\geq 0$.
For the log-convexity, we may assume $\theta\not=0,1$. 
Let $\kappa=\cF_G^{-1}\sigma \in L^2(\Gh)$.
We have
\begin{eqnarray*}
\|\sigma\|_{H^s,\omega_s}^2
=
\|\omega_s \kappa\|_{L^2(G)}^2
=
\|(\omega_{s_1}\kappa)^{2\theta} (\omega_{s_2}\kappa)^{2(1-\theta)}\|_{L^1(G)}
\\
\leq
\|\omega_{s_1}^{2\theta}\kappa\|_{L^p(G)}
\|\omega_{s_2}^{2(1-\theta)}\kappa\|_{L^q(G)},
\end{eqnarray*}
by H\"older's inequality with  $p=1/\theta$ and $q=1/(1-\theta)$.
\end{proof}

The difference operators are continuous on the Sobolev spaces:
\begin{lemma}
\label{lem_HsGh_Delta}
Let $s \geq0$.
Let $q$ be a continuous function on $G$ such that $q/\omega_s^{d/s}$ is bounded 
where $d\geq 0$ and  $\omega_s$ is a continuous function satisfying \eqref{eq_omega_s}.
Then $\Delta_q$ maps continuously $H^{s+d}(\Gh)$
to $H^s(\Gh)$:
$$
\exists C>0\qquad
\forall \sigma \in H^{s+d}(\Gh)
\qquad
\|\Delta_{q}\sigma\|_{H^{s}} \leq  
C
\|\sigma\|_{H^{s+d}} .
$$
An example of such a function $q$ is  any $d$-homogeneous polynomial. 
In particular 
$$
\|\Delta_{x^\alpha}\sigma\|_{H^{s}} \leq  
C
\|\sigma\|_{H^{s+[\alpha]}} .
$$
\end{lemma}

\begin{proof}[Proof of Lemma \ref{lem_HsGh_Delta}]
We have
\begin{eqnarray*}
\|\Delta_{q}\sigma\|_{H^{s}}
=
\|q \omega_s \cF_G^{-1}\sigma\|_{L^2(G)}
\leq 
\| q / { \omega_s^{d/s}}\|_{L^\infty(G)}
\|\omega_s^{\frac ds +1} \cF_G^{-1}\sigma\|_{L^2(G)},
\\
=
\| q / { \omega_s^{d/s}}\|_{L^\infty(G)}
\|\sigma\|_{H^s, \omega'_{s'}},
\end{eqnarray*}
where $\omega'_{s'}$ is the continuous function $\omega_s^{\frac ds +1} $
which satisfies \eqref{eq_omega_s} with $s'=d+s$.
\end{proof}

 The Sobolev spaces with integer exponents admit an equivalent description:
 \begin{lemma}
\label{lem_HsGh_integer}
If $s$ is a common multiple of $\upsilon_1,\ldots,\upsilon_n$, i.e. $s\in \nu_o\bN$, then  \begin{eqnarray*}
\sigma\in H^s(\Gh)
&\Longleftrightarrow&
\forall \alpha\in \bN_0^n, \ [\alpha]\leq s,
\quad
 \Delta_{x^\alpha} \sigma \in L^2(\Gh).
\end{eqnarray*}
Moreover 
$\sum_{[\alpha]\leq s}
\|\Delta_{x^\alpha}  \cdot\|_{L^2(\Gh)}$
is  an equivalent norm on $H^s(\Gh)$.
\end{lemma}

\begin{proof}[Proof of Lemma \ref{lem_HsGh_integer}]
Let $s\in \nu_o\bN$.
We consider the quasi-norm is $|\cdot|=|\cdot|_{\nu_o}$ given by
\eqref{eq_quasinormnuo} and the continuous function $\omega_s=
(1+|\cdot|^{2\nu_o})^{\frac s{2\nu_o}}$ which satisfies \eqref{eq_omega_s}.
Then $\omega_s^2$ is a polynomial in $x$, and more precisely a linear combination of squared  monomials:
$$
\omega_s^2(x)
=\sum_{[\alpha]\leq s} c_\alpha (x^{\alpha})^2
$$
for  some coefficients $(c_\alpha)$ depending on $s$ and $G$.
Thus
\begin{eqnarray*}
&&\|\sigma\|_{H^s,\omega_s}^2
=
\|\omega_s \cF_G^{-1}\sigma\|^2_{L^2(G)}
=\int_G 
\big|\sum_{[\alpha]\leq s} c_\alpha (x^{\alpha})^2\big| 
|\cF_G^{-1}\sigma (x)|^2 dx \\
&&\quad\leq
\sum_{[\alpha] \leq s} |c_\alpha| \int_G \big| x^\alpha \cF_G^{-1}\sigma (x)|^2 dx 
\leq
C \sum_{[\alpha] \leq s} \|\Delta_{x^\alpha}  \sigma\|_{L^2(\Gh)}^2.
\end{eqnarray*}
We have obtained
$$
\|\sigma\|_{H^s,\omega_s}
\leq C \sum_{[\alpha] \leq s} \|\Delta_{x^\alpha}  \sigma\|_{L^2(\Gh)}.
$$
The reverse inequality follows from Lemma \ref{lem_HsGh_Delta}.
\end{proof}

\begin{remark}
\begin{itemize}
\item 
In Lemma \ref{lem_HsGh_integer}, 
$(x^\alpha)$ may be replaced by any basis of homogeneous polynomials.
\item 
The hypothesis of divisibility of $s$ by $\upsilon_1,\ldots,\upsilon_n$
can not be removed in Lemma \ref{lem_HsGh_integer}. 
Indeed 
let us fix an index $\ell =1,\ldots,n$, and construct a sequence of symbols $\sigma_k$, $k\in\bN$, 
via $\cF_G^{-1} \sigma_k (x)=1_{|x_\ell-k|<1} \prod_{j\not=\ell} 1_{|x_j|<1}$.
One checks easily that 
$$
\|\sigma_k\|_{H^s}
\asymp 
k^s
\quad\mbox{but}\quad
\sum_{[\alpha]\leq s}
\|\Delta_{x^\alpha}  \sigma_k\|_{L^2(\Gh)}
\asymp
\sum_{[\alpha]\leq s} k^{\alpha_\ell}.
$$
If $s$ is a positive integer which is not divisible by $\upsilon_\ell$ then 
$k^{-s}\sum_{[\alpha]\leq s} k^{\alpha_\ell}  \to 0$ as $k\to\infty$.
\end{itemize}
\end{remark}

The following analogue of the Sobolev embedding holds
as an easy consequence of Corollary \ref{cor_prop_polar_coord}
with \eqref{eq_Fnorm_L1}:
\begin{lemma}
\label{lem_HsGh_sob_embedding}
 If $\sigma\in H^s(\Gh)$ with $s>Q/2$ then 
$\sigma\in \cF_G L^1(G)$ and
$$
\sup_{\pi\in \Gh} \|\sigma\|_{\sL(\cH_\pi)} \leq
\|\cF_G^{-1}\sigma\|_{L^1(G)}
\leq  C\|\sigma\|_{H^s}.
$$
\end{lemma}

As in the Euclidean case, we obtain an algebra for `point-wise multiplication' in the following sense:
\begin{lemma}
\label{lem_HsGh_algebra}
For any $\sigma$ and $\tau$ in $H^s(\Gh)$, 
the product $\sigma\tau=\{\sigma(\pi)\tau(\pi), \pi\in \Gh\}$  satisfies
(with possibly unbounded norms)
$$
\|\sigma \tau\|_{H^s}
\leq C
\left(\|\sigma\|_{H^s} 
\|\cF_G^{-1}\tau\|_{L^1(G)}
+
\|\cF_G^{-1}\sigma\|_{L^1(G)}
\|\tau\|_{H^s}\right),
$$ 
with a constant $C>0$ independent of $\sigma$ and $\tau$.

Hence for $s>Q/2$, 
if $\sigma, \tau\in H^s(\Gh)$
then 
$\sigma\tau\in H^s(\Gh)$,
and $H^s(\Gh)$ is a (non-commutative) algebra.
\end{lemma}

Note that Lemma \ref{lem_L1omegas} only yields
\begin{equation}
\label{eq_cq_lem_L1omegas} 
\|\sigma \tau\|_{H^s}
\lesssim \|\cF_G^{-1} \tau\|_{L^1(\omega_s)}
\|\sigma\|_{H^s}
\end{equation}
when a quasi-norm $|\cdot|$ and $\omega_s=(1+|\cdot|)^s$
with $s\geq 0$ have been fixed.
This does not prove Lemma \ref{lem_HsGh_algebra}.

 \begin{proof}[Proof of Lemma \ref{lem_HsGh_algebra}]
We fix a quasi-norm $|\cdot|$ and $\omega_s=(1+|\cdot|)^s$.
As a quasi-norm satisfies a triangular inequality (see Proposition \ref{prop_homogeneous_quasi_norm}), 
one checks easily
\begin{equation}
\label{eq_omegasxy+}
\exists C=C_{s,|\cdot|}\quad\forall x,y\in G\quad
\omega_s(x)\leq C 
\left(
\omega_s(xy^{-1})
+\omega_s(y)\right).
\end{equation}

Let $\sigma,\tau \in H^s(\Gh)$
and $f:=\cF_G^{-1}\sigma$, $g:=\cF_G^{-1}\tau$.
Then 
$$
\|\sigma\tau\|_{H^s,\omega_s}
=
\| \omega_s\ g*f\|_{L^2(G)}.
$$
The inequality in \eqref{eq_omegasxy+} implies
$$
\omega_s\ |g*f| \leq C \left(
(\omega_s |g|)*|f|+  |g|*(\omega_s|f|) \right),
$$
thus we obtain 
\begin{eqnarray*}
\| \omega_s\ g*f\|_{L^2(G)}
&\leq&
 C 
\left(\|(\omega_s |g|)*|f|\|_{L^2(G)}+ \| |g|*(\omega_s|f|)\|_{L^2(G)}
\right)
\\
&\leq&
 C 
\left(\|\omega_s |g|\|_{L^2(G)}\|f\|_{L^1(G)}
+ \|g\|_{L^1(G)}\|\omega_s f \|_{L^2(G)}
\right),
\end{eqnarray*}
by Young's inequality (see \eqref{eq_young}).
With Lemma \ref{lem_HsGh_sob_embedding},
the statement follows easily.
\end{proof}

\section{The Mihlin-H\"ormander condition on $\Gh$}
\label{sec_MHcond}

In the Euclidean case, 
the Mihlin-H\"ormander condition 
which implies that a function is an $L^p$-multiplier for all $p>1$
is the membership in Sobolev spaces locally uniformly, see the introduction.
The aim of this section is to define the membership in Sobolev spaces locally uniformly in our context and express our main multiplier theorem in term of this membership.
This requires first to define dilations on $\Gh$.

\subsection{Dilations on $\Gh$}
\label{subsec_dilation_Gh}

In this section, we define dilations on the set $\Gh$.
This is possible thanks to the following lemma 
whose proof is a routine exercise of representation theory:
\begin{lemma}
\begin{enumerate}
\item 
If $\pi$ is a unitary irreducible representation of $G$
and $r>0$
then setting
\begin{equation}
\label{eq_def_rpi}
r\cdot \pi (x)
=
\pi(r x)
\quad,\quad x\in G,
\end{equation}
we have defined a unitary irreducible representation $r\cdot \pi$ of $G$.
\item 
If $\pi_1$ and $\pi_2$ are two equivalent unitary irreducible representations of $G$,
then, for any $r>0$,
$r\cdot \pi_1$ and $r\cdot \pi_2$ are  two equivalent unitary irreducible representations of $G$.
\end{enumerate}
\end{lemma}

\begin{definition}
For any $\pi\in \Gh$ and any $r>0$, 
Equation \eqref{eq_def_rpi} defines a new class 
$$
r\cdot \pi:=D_r (\pi) \in \Gh.
$$
\end{definition}

These dilations define an  action of $\bR^{*+}$ on $\Gh$
which interacts nicely with the group structure:
\begin{lemma}
\label{lem_rpiXalpha+kappa}
Let $\pi\in \Gh$ and $r>0$.

For any $\alpha\in \bN^n$,
$$
r\cdot \pi (X^\alpha) = r^{[\alpha]} \pi(X^\alpha).
$$

For any positive Rockland operator $\cR$ of degree $\nu_\cR$,
$$
r\cdot \pi (\cR) =r^{\nu_\cR} \pi(\cR)
$$
and if $f\in L^\infty(\bR)$ then (spectral definitions)
$$
f(r\cdot \pi(\cR))=f(r^{\nu_\cR} \pi(\cR)).
$$

If $\kappa\in L^2(G)\cup L^1(G)$ then
$$
(r\cdot \pi) (\kappa) = \pi ( r^{-Q} \kappa(r^{-1}\cdot)).
$$
Consequently,
$$
\Delta^\alpha \{\widehat \kappa 
(r\cdot \pi)\} = 
r^{[\alpha]}
\left(\Delta^\alpha \widehat \kappa \right) (r\cdot \pi).
$$
\end{lemma}
The proof of Lemma \ref{lem_rpiXalpha+kappa} is left to the reader.

The Sobolev spaces on $\Gh$ are invariant under these dilations:
\begin{lemma}
\label{lem_dilation_Hs}
Let $\sigma \in L^2(\Gh)$ and $s\geq 0$.
If $\sigma\in H^s(\Gh)$ then 
$\sigma\circ D_r = \{\sigma(r\cdot\pi),\pi\in \Gh\}$
is in  $H^s(\Gh)$
for  all $r>0$.
Furthermore 
let us fix  a quasi-norm $|\cdot|$ and  $\omega_s=(1+|\cdot|)^s$.
For every $r>0$ and $\sigma\in L^2(\Gh)$, we have
$$
\|\sigma\circ D_r\|_{H^s,\omega_s}
\leq 
(1+r)^s r^{-\frac Q2}
\|\sigma\|_{H^s,\omega_s}.
$$
\end{lemma}
\begin{proof}[Proof of Lemma \ref{lem_dilation_Hs}]
 Lemma \ref{lem_rpiXalpha+kappa} and the change of variable $D_r$ yield
$$
\|\sigma \circ D_r\|_{H^s, \omega_{s}}
=
\|\omega_s r^{-Q} (\cF_G^{-1}\sigma)\circ D_{r^{-1}}\|_{L^2}
=
r^{-\frac Q2} 
\| (1+r|\cdot|)^s  \cF_G^{-1}\sigma\|_{L^2}.
$$
We conclude with 
$ (1+r|\cdot|)^s   \leq (1+r)^s\omega_s$.
\end{proof}

\subsection{Fields locally uniform in $H^s(\Gh)$}
The aim of this section is to define and study the Banach space of fields locally uniformly in $H^s(\Gh)$.
In the Euclidean case, 
the membership in Sobolev spaces locally uniformly
 is the Mihlin-H\"ormander condition
which implies that a function is an $L^p$-multiplier for all $p>1$.
This motivates the following definition.

\begin{definition}
\label{def_Hslu}
Let $s\geq 0$.
We say that  a measurable field of operators $\sigma=\{\sigma(\pi), \pi\in \Gh\}$
is \emph{uniformly locally} in $H^s(\Gh)$
on the right, respectively on the left,
when there exist a positive Rockland operator $\cR$ and 
a non-zero function $\eta\in \cD(0,\infty)$  such that
the quantity
\begin{equation}
\label{eq_def_HsluR_norm}
\|\sigma\|_{H^s_{l.u.,R},\eta,\cR}:=
\sup_{r>0} \|\{\sigma (r\, \cdot\pi)\ \eta (\pi(\cR)), \pi\in \Gh\}\|_{H^s} 
\end{equation}
or respectively
\begin{equation}
\label{eq_def_HsluL_norm}
\|\sigma\|_{H^s_{l.u.,L},\eta,\cR}:=
\sup_{r>0} \| \{\eta (\pi(\cR)) \ \sigma (r\cdot\pi), \pi\in \Gh\}\|_{H^s},
\end{equation}
is finite.
\end{definition}

Our first task will be to show that, as in the Euclidean case, 
this definition does not depend on the cut-off function.
Here we also have to prove that it does not depend on the Rockland operator.
This is the object of the following statement which will be proved in 
Section \ref{subsec_pf_prop_Hslu_norm}.
\begin{proposition}
\label{prop_Hslu_norm}
Let 
$\sigma=\{\sigma(\pi), \pi\in \Gh\}$
be  a measurable field of operators such that 
$\|\sigma\|_{H^s_{l.u.,R},\eta,\cR}$ is finite for some  positive Rockland operator $\cR$ and $\eta\in \cD(0,\infty)\setminus \{0\}$.
Then  for any positive Rockland operator $\cS$ and 
any function $\zeta\in \cD(0,\infty)$,  
the quantity $\|\sigma\|_{H^s_{l.u.,R},\zeta,\cS}$ is finite 
and there exists a constant $C>0$ 
(depending on $\cR,\cS$ and $\eta,\zeta$ but not on $\sigma$)
such that
$$
\|\sigma\|_{H^s_{l.u.,R},\zeta,\cS}
\leq 
C \|\sigma\|_{H^s_{l.u.,R},\eta,\cR}.
$$
\end{proposition}
We have a similar result for the left case, 
and we denote by $H^s_{l.u.,R}(\Gh)$, resp. $H^s_{l.u.,L}(\Gh)$,
the space of measurable fields which are
 uniformly locally in $H^s(\Gh)$
on the right, respectively on the left.
Furthermore these spaces are Banach spaces
with the following properties:
\begin{corollary}
\label{cor_prop_Hslu_norm}
\begin{enumerate}
\item 
\label{item_cor_prop_Hslu_norm_banach}
If $s\geq 0$, 
the space $H^s_{l.u.,R}(\Gh)$ is a Banach space when equipped with any equivalent norm $\|\cdot\|_{H^s_{l.u.,R},\eta,\cR}$, 
where $\eta\in \cD(0,\infty)$ is non-zero and $\cR$ is a positive Rockland operator.
\item 
\label{item_cor_prop_Hslu_norm_inclusion}
We have the continuous inclusion
$$
H^{s_1}_{l.u.,R}(\Gh) \subset H^{s_2}_{l.u.,R}(\Gh),
\qquad  s_1\geq s_2.
$$
\item 
\label{item_cor_prop_Hslu_norm_dilation}
If $\sigma \in H^s_{l.u.,R}(\Gh)$ and $r_o>0$, 
then 
$\sigma\circ D_{r_o}\in H^s_{l.u.,R}(\Gh)$ satisfies
$$
\|\sigma \circ D_{r_o}\|_{H^s_{l.u.,R},\eta,\cR}
=
\|\sigma\|_{H^s_{l.u.,R},\eta(r_o^{-1}\cdot ),\cR}.
$$
\item 
\label{item_cor_prop_Hslu_norm_sob}
If $s>Q/2$, 
we have the continuous inclusion of Sobolev type:
$$
H^s_{l.u.,R} \subset L^\infty(\Gh)
$$
\end{enumerate}
We have similar statements for the left case.
\end{corollary}

This corollary will  be proved in 
Section \ref{subsec_pf_cor_prop_Hslu_norm}.
We can already point out that taking the adjoint provides the link between the left and right cases:
\begin{lemma}
\label{lem_Hslu_LR}
Let  $\sigma=\{\sigma(\pi),\pi\in \Gh\}$ be
a $\mu$-measurable field of operators and $s\geq 0$.
Then 
$$
\sigma \in 
H^s_{l.u.,R}(\Gh)
\Longleftrightarrow
\sigma^* \in 
H^s_{l.u.,L}(\Gh),
$$
and in this case 
$$
\|\sigma\|_{H^s_{l.u.,R},\eta,\cR}
=
\|\sigma^*\|_{H^s_{l.u.,L},\eta,\cR}.
$$
We can reverse the r\^ole of left and right.
\end{lemma}

 Lemma \ref{lem_Hslu_LR} follows readily from
\eqref{eq_sigma*Hs}.

\medskip

The following statement gives sufficient conditions for the membership 
in $H^s_{l.u.,R}(\Gh)$ and $H^s_{l.u.,L}(\Gh)$.
\begin{proposition}
\label{prop_Hslu_suff_cond}
Let  $\sigma=\{\sigma(\pi),\pi\in \Gh\}$ be
a $\mu$-measurable field of operators and $s\geq 0$.
Let $\cR$ be a positive Rockland operator 
and let $\eta\in \cD(0,\infty)$ be non-zero.

\begin{description}
\item[(Left)] 
If   $\pi(\cR)^{\frac{[\alpha]}\nu}\Delta^\alpha \sigma \in L^\infty(\Gh)$ for all $|\alpha|\leq N$ for some positive integer $N\in \bN$ divisible by $\upsilon_1,\ldots,\upsilon_n$,
then $\sigma \in H^N_{l.u.,L}(\Gh)$ 
and 
$$
\|\sigma\|_{H^N_{l.u.,L},\eta,\cR}
\leq 
C
 \sum_{[\alpha]\leq N}
 \sup_{\pi\in \Gh}
 \|\pi(\cR)^{\frac{[\alpha]}\nu}\Delta^\alpha\sigma(\pi)\|_{\sL(\cH_\pi)},
$$
where the constant $C>0$ does not depend on $\sigma$.
\item[(Right)] 
If   $\Delta^\alpha \sigma\ \pi(\cR)^{\frac{[\alpha]}\nu} \in L^\infty(\Gh)$ for all $|\alpha|\leq N$ for some positive integer $N\in \bN$ divisible by $\upsilon_1,\ldots,\upsilon_n$,
then $\sigma \in H^N_{l.u.,R}(\Gh)$ 
and 
$$
\|\sigma\|_{H^N_{l.u.,R},\eta,\cR}
\leq 
C
 \sum_{[\alpha]\leq N}
 \sup_{\pi\in \Gh}
 \|\Delta^\alpha\sigma(\pi)\ \pi(\cR)^{\frac{[\alpha]}\nu}\|_{\sL(\cH_\pi)},
$$
where the constant $C>0$ does not depend on $\sigma$.
\end{description}
\end{proposition}

 Proposition \ref{prop_Hslu_suff_cond} will be shown 
 in Section \ref{subsec_pf_prop_Hslu_suff_cond}.
Note that in this statement above, 
the meaning of $\Delta^\alpha\sigma(\pi)$ requires the slightly more general definition of difference operator alluded to in Remark
\ref{rem_def_Delta}.

\begin{remark}
\label{rem_prop_Hslu_suff_cond}
The suprema in Proposition \ref{prop_Hslu_suff_cond}
are independent of a choice of a positive Rockland operator, 
see Propositions \ref{prop_powercR1cR2}
and \ref{prop_Hslu_norm}.
Moreover, the condition described in Proposition \ref{prop_Hslu_suff_cond}
is invariant under dilation by Part \eqref{item_cor_prop_Hslu_norm_dilation}
of Corollary \ref{cor_prop_Hslu_norm}
and for  the suprema involved in Proposition \ref{prop_Hslu_suff_cond}, 
 by Lemma \ref{lem_rpiXalpha+kappa} and the following calculations:
 \begin{eqnarray*}
&&\pi(\cR)^{\frac{[\alpha]}\nu}\Delta^\alpha( \sigma\circ D_{r_o}) (\pi)
=
{r_o}^{[\alpha]} \ \pi(\cR)^{\frac{[\alpha]}\nu} \Delta^\alpha( \sigma) ({r_o}\cdot \pi)
\\
&&\qquad
=({r_o}\cdot \pi)(\cR)^{\frac{[\alpha]}\nu} \Delta^\alpha( \sigma) ({r_o}\cdot \pi)
=
\pi_1(\cR)^{\frac{[\alpha]}\nu}\Delta^\alpha \sigma  (\pi_1),
\end{eqnarray*}
with $\pi_1={r_o}\cdot \pi$.
Therefore
\begin{equation}
\label{eq_prop_Hslu_suff_cond_dilation}
\sup_{\pi\in \Gh}
 \|\pi(\cR)^{\frac{[\alpha]}\nu}\Delta^\alpha(\sigma\circ D_{r_o}) (\pi)\|_{\sL(\cH_\pi)}
 =
 \sup_{\pi_1\in \Gh}
 \|\pi_1(\cR)^{\frac{[\alpha]}\nu}\Delta^\alpha\sigma(\pi_1)\|_{\sL(\cH_{\pi_1})}
 .
\end{equation}
\end{remark}

\subsection{The main result}

The main result of this article is Theorem \ref{thm_main0}
which we now rephrase as:
\begin{theorem}
\label{thm_main}
Let $G$ be a graded nilpotent Lie group.
If  $\sigma=\{\sigma(\pi),\pi\in \Gh\}\in H^s_{l.u.,R}(\Gh) \cap H^s_{l.u.,L}(\Gh) $
for some $s>Q/2$
then the corresponding operator $T=T_\sigma$ is bounded on $L^p(G)$ for any $1<p<\infty$.
Furthermore, 
$$
\|T\|_{\sL(L^p(G))}
\leq C 
\max\left(\|\sigma\|_{H^s_{l.u.,R},\eta,\cR},\|\sigma\|_{H^s_{l.u.,L},\eta,\cR}\right),
$$
where $C>0$ is a constant independent of $\sigma$
but may depend on $p,s,G$ and a choice of $\eta \in \cD(0,\infty)$ and a
positive Rockland operator $\cR$.
\end{theorem}

By Proposition \ref{prop_Hslu_suff_cond},
 Theorem \ref{thm_main} implies 
Theorem \ref{thm_intro}.

The hypotheses and the conclusion of Theorems \ref{thm_main} and  \ref{thm_intro} are 
`dilation-invariant' and do not depend of a choice  of a Rockland operator
or a function $\eta$, 
see Remark \ref{rem_prop_Hslu_suff_cond} and Corollary \ref{cor_prop_Hslu_norm}.

Theorem \ref{thm_main} is proved in Section \ref{subsec_pf_thm_main}
and its proof yields  the following more precise version:
\begin{corollary}
\label{cor_thm_main}
Let $G$ be a graded nilpotent Lie group.
Let  $\sigma=\{\sigma(\pi),\pi\in \Gh\}$ be
a $\mu$-measurable field of operators in $L^2(\Gh)$
and let $T_\sigma$ be the corresponding Fourier multiplier operator on $\cS(G)$.
\begin{enumerate}
\item 
If $\sigma$ is in $H^s_{l.u.,R}$ or $H^s_{l.u.,L}$
for some $s>Q/2$, 
then $T$ is bounded on $L^2(G)$
with 
$$
\|T\|_{\sL(L^2(G))}= \sup_{\pi\in \Gh} \|\sigma(\pi)\|_{op}
\leq C_2 
\left\{\begin{array}{l}
\|\sigma\|_{H^s_{l.u.,R}}\\
\mbox{or}\\
\|\sigma\|_{H^s_{l.u.,L}}\\
\end{array}\right. ,
$$
respectively, and $C_2$ a constant independent of $\sigma$.
\item 
If $\sigma\in H^s_{l.u.,R}$
for some $s>Q/2$ 
then $T$ is of weak-type $L^1$.
Moreover there exists a constant $C_1>0$ independent of $\sigma$, such that
$$
\forall f\in \cS(G)\quad \forall \alpha>0\quad
|\{x: |Tf(x)|>\alpha\} \leq C_1\frac { \|\sigma\|_{H^s_{l.u.,R}}} {\alpha} 
\|f\|_{L^1(G)}.
$$
For each $p\in (1,2)$, 
there exists a constant $C_p>0$ independent of $\sigma$, such that
$$
\forall f\in \cS(G)\quad 
\|Tf\|_{L^p}\leq C_p\|\sigma\|_{H^s_{l.u.,R}} \|f\|_{L^p(G)}.
$$

\item 
If $\sigma\in H^s_{l.u.,L}$
for some $s>Q/2$ 
then $T^*$ is of weak-type $L^1$.
Moreover there exists a constant $C_1>0$ independent of $\sigma$, such that
$$
\forall f\in \cS(G)\quad \forall \alpha>0\quad
|\{x: |T^*f(x)|>\alpha\} \leq C_1  \frac {\|\sigma\|_{H^s_{l.u.,L}}} {\alpha} 
\|f\|_{L^1(G)}.
$$
For each $p\in (2,\infty)$, 
there exists a constant $C_p>0$ independent of $\sigma$, such that
$$
\forall f\in \cS(G)\quad 
\|Tf\|_{L^p}\leq C_p\|\sigma\|_{H^s_{l.u.,L}} \|f\|_{L^p(G)}.
$$
\end{enumerate}
\end{corollary}

In the statement above, $\|\sigma\|_{H^s_{l.u.,R}} $
denotes a choice of norms $\|\sigma\|_{H^s_{l.u.,R},\cR,\eta} $
and similarly of the left case.
The constants in the statement depends on this choice.

\section{Proofs}
\label{sec_proofs}

Here we give the proofs of earlier statements:
Proposition \ref{prop_Hslu_norm} in Section \ref{subsec_pf_prop_Hslu_norm}, 
 Corollary \ref{cor_prop_Hslu_norm}
 in Section \ref{subsec_pf_cor_prop_Hslu_norm},
 Proposition \ref{prop_Hslu_suff_cond}
 in Section \ref{subsec_pf_prop_Hslu_suff_cond}
 
 \subsection{Proof of Proposition \ref{prop_Hslu_norm} }
\label{subsec_pf_prop_Hslu_norm}

Let $\eta$ and $\cR$ be fixed as in Proposition \ref{prop_Hslu_norm}.
We may assume $\eta$ real valued
(otherwise we consider separately $\Re \eta$ and $\Im \eta$).
Let $c_o>0$ such that  $2^{c_o}I$ intersects $I$
where  $I$ is an open interval inside the support of $\eta$.
For $\lambda\in \bR$ and $j\in \bZ$, we set
$$
\eta_j(\lambda) =\eta(2^{-c_o j} \lambda)
\quad\mbox{and}\quad
\alpha(\lambda):=\sum_{j\in \bZ} \eta_j^2(\lambda).
$$
One checks easily that $\alpha$ is constantly 0 on $(-\infty,0]$ and  that it is smooth and valued in $(0,+\infty)$ on $(0,+\infty)$.
Furthermore, 
$$
\forall \lambda\in \bR, \ j\in \bZ\quad
\alpha(2^{jc_o} \lambda)=\alpha(\lambda),
\qquad\mbox{and}\qquad
\forall \lambda>0\quad
\sum_{j\in \bZ} \frac{\eta_j^2}\alpha (\lambda)
=
1,
$$
and $\id_{\cH_\pi}=\sum_{j\in \bZ} \frac{\eta_j^2}\alpha (\pi(\cR))$
with convergence in the Strong Operator Topology of  $\sL(\cH_\pi)$.
Hence, we have
$$
\|\sigma(r\cdot \pi) \zeta(\pi(\cS))\|_{H^s}
\leq
\sum_{j\in \bZ} E_j
\qquad\mbox{where}\qquad
E_j:=\|\sigma(r\cdot \pi) \frac{\eta_j^2}\alpha (\pi(\cR))  \zeta(\pi(\cS))\|_{H^s}.
$$

\noindent \textbf{Case $j\geq 0$:}
By \eqref{eq_cq_lem_L1omegas}, we have 
$$
E_j \lesssim \|\sigma(r\cdot \pi) \eta_j(\pi(\cR))\|_{H^s}
\|\cF_G^{-1}\frac{\eta_j}\alpha (\pi(\cR))  \zeta(\pi(\cS))\|_{L^1(\omega_s)}.
$$
Lemma \ref{lem_rpiXalpha+kappa} yields
$\|\sigma(r\cdot \pi) \eta_j(\pi(\cR))\|_{H^s} \lesssim 2^{jc_0 Q/(2\nu_\cR)} \|\sigma\|_{H^s_{l.u.,R},\eta,\cR}$, while
by Lemma \ref{lem_L1omegas}
the $L^1(\omega_s)$-norm is 
\begin{eqnarray*}
&\lesssim
\|\cF_G^{-1}\left\{\frac{\eta_j}{\alpha }(\pi(\cR))  \pi(\cR)^{-N}\right\}\|_{L^1(\omega_s)}
\|\cF_G^{-1}\left\{ \pi(\cR)^N  \zeta(\pi(\cS))\right\}\|_{L^1(\omega_s)}
\\
&=
2^{-jc_0N}\|\frac{\lambda^{-N} \eta_0}\alpha (2^{-jc_0}\cR)  \delta_0 \|_{L^1(\omega_s)}
\|\cR^N  \zeta(\cS)\delta_0\|_{L^1(\omega_s)}
\end{eqnarray*}
for a suitable positive integer $N$.
By Hulanicki's Theorem (Theorem \ref{thm_hula}), 
$\zeta(\cS)\delta_0 \in \cS(G)$ so the second $L^1(\omega_s)$-norm is finite, and the first one is 
$\lesssim 2^{jc_0 d}$ for some $d\in \bN$ which depends on $s,\cR, G$ but not on $N$.
Hence we have obtained 
$E_j\lesssim 2^{jc_0(d -N +Q/2)} \|\sigma\|_{H^s_{l.u.,R},\eta,\cR}$
and, choosing an integer $N$ such that $N>d+Q/2$, 
we have $\sum_{j\geq 0} E_j\lesssim \|\sigma\|_{H^s_{l.u.,R},\eta,\cR}$.

\smallskip

\noindent \textbf{Case $j< 0$:}
By  Lemma \ref{lem_rpiXalpha+kappa} and \eqref{eq_cq_lem_L1omegas},
we have
$$
E_j\lesssim 
2^{-\frac {j c_o}{\nu_\cR}(s-\frac Q2)}
\|\sigma\|_{H^s_{l.u.,R},\cR,\eta} 
\left\|\cF_G^{-1}\left\{\frac {\eta}{\alpha} (\pi(\cR))
\zeta( 2^{-j \frac{c_o\nu_\cS}{\nu_\cR}}\cdot \pi(\cS ))\right\}\right\|_{L^1( \omega_{s}) }
$$
By Lemma \ref{lem_L1omegas}, 
the $L^1( \omega_{s})$-norm is
\begin{eqnarray*}
\lesssim
	\left\|\cF_G^{-1}\left\{\frac {\eta}{\alpha} (\pi(\cR))
\pi(\cS)^{N'}\right\}\right\|_{L^1( \omega_{s}) }
\left\|\cF_G^{-1}\left\{\pi(\cS)^{-N'}\zeta( 2^{-j \frac{c_o\nu_\cS}{\nu_\cR}}\cdot \pi(\cS ))\right\}\right\|_{L^1( \omega_{s}) }
\\
=
\left\|\tilde \cS^{N'}\frac {\eta}{\alpha} (\cR)\delta_0\right\|_{L^1( \omega_{s}) }
2^{jN' \frac{c_o}{\nu_\cR}}
\left\|(\lambda^{-N'}\zeta)( 2^{j \frac{c_o}{\nu_\cR}} \cS )\right\|_{L^1( \omega_{s}) }
\end{eqnarray*}
for a suitable positive integer $N'$.
By Hulanicki's Theorem (Theorem \ref{thm_hula}), 
$\frac {\eta}{\alpha} (\cR)\delta_0\in \cS(G)$ so the first $L^1(\omega_s)$-norm is finite, and the second one is 
$\lesssim 2^{|j|\frac{c_0}{\nu_\cR} d'}$ for some $d'\in \bN$ which depends on $s,\cS, G$ but not on $N'$.
Hence we have obtained 
$E_j\lesssim 2^{j\frac{c_0}{\nu_\cR}(-d' +N' +Q/2-s)} \|\sigma\|_{H^s_{l.u.,R},\eta,\cR}$
and, choosing an integer $N'$ such that $N'>d'-Q/2+s$, 
$\sum_{j<0 } E_j\lesssim \|\sigma\|_{H^s_{l.u.,R},\eta,\cR}$.

This concludes the proof of Proposition \ref{prop_Hslu_norm}.

\subsection{Proof of  Corollary \ref{cor_prop_Hslu_norm}}
\label{subsec_pf_cor_prop_Hslu_norm}

If $\|\sigma\|_{H^s_{l.u.,R},\eta,\cR}=0$, 
then by Lemma \ref{lem_dilation_Hs},
 $$\|\sigma \eta(r \pi(\cR))\|_{H^s}=0,$$ 
and the field $\sigma(\pi) \eta(r \pi(\cR))$ is identically zero for any $r>0$,  since $H^s(\Gh)$ is a normed space.
Choosing $\eta$ such that e.g. $\eta\equiv1 $ on $[1,2]$,
this implies that for any $a,b \geq 0$, 
$\sigma (\pi) E_\pi [a,b]\equiv0$ where $E_\pi$ is the spectral resolution of $\pi(\cR)$ (or equivalently the group Fourier transform of the spectral resolution of $\cR$, see \cite{FR-monograph}).
Hence  $\sigma =0$. 

Let $\{\sigma_\ell\}$ be a Cauchy sequence in $H^s_{l.u.,R}(\Gh)$,
that is, 
\begin{equation}
\label{eq_cor_prop_Hslu_norm}
\forall\epsilon>0 \quad\exists \ell_\epsilon \in \bN \ : \
\forall \ell_1,\ell_2\geq \ell_\epsilon \ \forall r>0\quad
\|(\sigma_{\ell_1}-\sigma_{\ell_2})(r\cdot\pi) \eta(\pi(\cR))\|_{H^s}\leq \epsilon.
\end{equation}
This implies that $\{\sigma_\ell(r\cdot\pi) \eta(\pi(\cR))\}$ is a Cauchy sequence in the Banach space $H^s(\Gh)$ for each $r>0$ fixed. 
For the same reasons as above, 
this shows that  $\{\sigma_\ell E_\pi [a,b]\}$ is a Cauchy sequence in the Banach space $H^s(\Gh)$.
Hence it converges towards a limit $\sigma^{([a,b])}$ in $H^s(\Gh)$
with $\sigma^{([a,b])} = \sigma^{([c,d])} E_\pi [a,b] $ if $[a,b]\subset [c,d]$.
This defines a field of operators $\sigma$ which satisfies
 for each $r>0$ fixed:
$$
\lim_{\ell\to\infty} \sigma_\ell(r\cdot\pi)\eta(\pi(\cR))=
\sigma(r\cdot\pi)\eta(\pi(\cR)).
$$
Passing to the limit in \eqref{eq_cor_prop_Hslu_norm},
this shows that $\sigma$ is also the limit of 
 $\{\sigma_\ell\}$  in $H^s_{l.u.,R}(\Gh)$.
 This shows Part \eqref{item_cor_prop_Hslu_norm_banach}  of  Corollary \ref{cor_prop_Hslu_norm}.

Part \eqref{item_cor_prop_Hslu_norm_inclusion} follows from 
the similar inclusions for $H^s(\Gh)$, 
see Lemma \ref{lem_HsGh_inclusion}.
Part \eqref{item_cor_prop_Hslu_norm_dilation} is easily checked.

It remains to show Part \eqref{item_cor_prop_Hslu_norm_sob}.
Let $s>Q/2$.
We may choose $\eta\in \cD(0,\infty)$ supported in $[1/2,4]$
such that  $0\leq \eta\leq 1$ and $\eta\equiv 1$ on $[1,2]$.
The properties of the functional calculus yield
\begin{eqnarray*}
&&\sup_{\pi\in \Gh} \|\sigma(\pi) E_\pi[r^{-1}, 2r^{-1}]\|_{\sL(\cH_\pi)} 
\leq
\sup_{\pi\in \Gh} \|\sigma(r\cdot\pi) E[1,2] \eta (\pi(\cR))\|_{\sL(\cH_\pi)} 
\\
&&\qquad\leq
\sup_{\pi\in \Gh} \|\sigma(r\cdot\pi)\ \eta (\pi(\cR))\|_{\sL(\cH_\pi)} 
\leq
C\|\sigma \|_{H^s_{l.u.,R},\eta,\cR}
\end{eqnarray*}
by the Sobolev embedding  of $H^s(\Gh)$, 
see Lemma \ref{lem_HsGh_sob_embedding}.
Here, the constant $C$ is independent of $r>0$, 
thus the supremum over $r>0$ of
$\sup_{\pi\in \Gh} \|\sigma(\pi) E_\pi[r^{-1}, 2r^{-1}]\|_{\sL(\cH_\pi)}$
is finite. This shows that $\sigma\in L^\infty(\Gh)$.
This also concludes the proof of Corollary \ref{cor_prop_Hslu_norm}.

\subsection{Proof of Proposition \ref{prop_Hslu_suff_cond}}
\label{subsec_pf_prop_Hslu_suff_cond}

We will prove the second statement for the right spaces.
We have already noted that the  statement requires the slightly more general definition of difference operator alluded to in Remark
\ref{rem_def_Delta}. It is also the case for this proof.

Let $N\in \nu_o\bN$, that is, a positive integer divisible by $\upsilon_1,\ldots,\upsilon_n$.
Let $\sigma \in L^\infty(\Gh)$ be such that
 $\Delta^\alpha \sigma\  \pi(\cR)^{\frac{[\alpha]}\nu}\in L^\infty(\Gh)$ for all $|\alpha|\leq N$.
By Lemma \ref{lem_HsGh_integer},
we have
\begin{eqnarray*}
\|\sigma(r\cdot\pi) \eta(\pi(\cR))\|_{H^N(\Gh)}
&\asymp&
\sum_{[\alpha]\leq N}
\|\Delta_{x^\alpha} \left(\sigma(r\cdot\pi) \eta(\pi(\cR)\right)\|_{L^2(\Gh)}
\\
&\lesssim&
\sum_{[\alpha_1]+[\alpha_2]\leq N}
\|
\Delta_{x^{\alpha_1}} (\sigma(r\cdot\pi)) \
\Delta_{x^{\alpha_2}} \eta(\pi(\cR)) \|_{L^2(\Gh)},
\end{eqnarray*}
by the Leibniz formula, see \eqref{eq_Leibniz_rule_sigma}.
Inserting powers of $\pi(\cR)$,
we have for each term above the estimate
$$
\|\Delta_{x^{\alpha_1}}( \sigma(r\cdot\pi)) \
\Delta_{x^{\alpha_2}} \eta(\pi(\cR)) \|_{L^2(\Gh)}
\leq
\|\Delta_{x^{\alpha_1}} (\sigma(r\cdot\pi) )\
\pi(\cR)^{\frac{[\alpha_1]}{\nu}} \|_{L^\infty(\Gh)}
\|\pi(\cR)^{-\frac{[\alpha_1]}{\nu}}
\Delta_{x^{\alpha_2}} \eta(\pi(\cR)) \|_{L^2(\Gh)}.
$$
For the first term, by \eqref{eq_prop_Hslu_suff_cond_dilation}, we have
$$
\|\Delta_{x^{\alpha_1}} (\sigma(r\cdot\pi))  \
\pi(\cR)^{\frac{[\alpha_1]}{\nu}} \|_{L^\infty(\Gh)}
=
\sup_{\pi_1\in \Gh}
 \|\Delta^\alpha\sigma(\pi_1)\ \pi_1(\cR)^{\frac{[\alpha]}\nu}\|_{\sL(\cH_{\pi_1})}.
$$
For the second term, we define $\eta_M\in \cD(0,\infty)$ via $\eta_M (\lambda) = \lambda^{-M} \eta(\lambda)$ 
for an integer $M\in\bN$ to be chosen.
Using again the Leibniz formula, we have
\begin{eqnarray*}
\|\pi(\cR)^{-\frac{[\alpha_1]}{\nu}}
\Delta_{x^{\alpha_2}} \eta(\pi(\cR)) \|_{L^2(\Gh)}
&\lesssim 
\sum_{[\alpha_3]+[\alpha_4]=[\alpha_2]}
	\|\pi(\cR)^{-\frac{[\alpha_1]}{\nu}} 
(\Delta_{x^{\alpha_3}} \pi(\cR)^M) \pi(\cR)^{\frac{[\alpha_1]}{\nu} -M\nu +\alpha_3}  \|_{L^\infty(\Gh)}
\\
&\qquad\quad \|\pi(\cR)^{-\frac{[\alpha_1]}{\nu} +M\nu -\alpha_3}  \Delta_{x^{\alpha_4}} \eta_M(\pi(\cR)) \|_{L^2(\Gh)}.
\end{eqnarray*}
By Hulanicki's Theorem, see Theorem \ref{thm_hula},
the function
$x^{\alpha_4} \ \eta_M(\pi(\cR))\delta_0$
is Schwartz.
We fix $M$ such that $-\frac{[\alpha_1]}{\nu} +M\nu -\alpha_3\geq 0$ for all $\alpha_3$ as above.
In this way, $x^{\alpha_4} \ \eta_M(\pi(\cR))\delta_0$ 
is in the domain of $\cR^{-\frac{[\alpha_1]}{\nu} +M\nu -\alpha_3}$ by Proposition \ref{prop_powercR1cR2}.
Hence $\|\pi(\cR)^{-\frac{[\alpha_1]}{\nu} +M\nu -\alpha_3}  \Delta_{x^{\alpha_4}} \eta_M(\pi(\cR)) \|_{L^2(\Gh)}$ is finite.
For the $L^\infty(\Gh)$ term, 
easy computations \cite[Lemma 5.2.9]{FR-monograph} show that 
$\Delta_{x^{\alpha_3}} \pi(\cR)^M$ is the image via $\cF$ of a homogeneous left-invariant differential operator $T$ 
of degree $M\nu_0 - [\alpha_3]$.
By \cite[Theorem 4.4.16]{FR-monograph}, 
$\cR^{-\frac{[\alpha_1]}{\nu}} T \cR^{\frac{[\alpha_1]}{\nu} -M\nu +\alpha_3} $ is bounded, 
thus $\|\pi(\cR)^{-\frac{[\alpha_1]}{\nu}} 
(\Delta_{x^{\alpha_3}} \pi(\cR)^M) \pi(\cR)^{\frac{[\alpha_1]}{\nu} -M\nu +\alpha_3}  \|_{L^\infty(\Gh)}$ is finite.
We have therefore obtained that 
$$
\|\sigma(r\cdot\pi) \eta(\pi(\cR))\|_{H^N(\Gh)}
\lesssim  \sum_{[\alpha_1]\leq N}
\sup_{\pi_1\in \Gh}
 \|\Delta^\alpha\sigma(\pi_1)\ \pi_1(\cR)^{\frac{[\alpha]}\nu}\|_{\sL(\cH_{\pi_1})}.
$$
Taking the supremum over $r$ on the left hand side
proves Proposition  \ref{prop_Hslu_suff_cond} for the condition on the right.
For the condition on the left, one can proceed in a similar way or obtain it  by taking the adjoint from the condition on the right, see Lemma \ref{lem_Hslu_LR}.

\subsection{Proof of Theorem \ref{thm_main}}
\label{subsec_pf_thm_main}

Let $\sigma\in H^s_{l.u.,R}$ with $s>Q/2$.
We want to show that the Fourier multiplier operator $T_\sigma$
admits an $L^p$-bounded extension.
We will follow the classical way to do this: we prove that $T_\sigma$ is a Calder\'on-Zygmund operator on the space of homogeneous type $G$, 
see \cite[Ch. III]{coifman+weiss}.

Let $\eta\in \cD(0,\infty)$  be supported in $[1/2,2]$,
valued in $[0,1]$ and satisfying 
$\sum_{j\in \bZ} \eta_j\equiv 1$ on $(0,\infty)$
where $\eta_j(\lambda)=\eta(2^{-j}\lambda)$.
For each $j\in \bZ$ and $\pi\in \Gh$,
we set
$$
\sigma_j(\pi)=\sigma(2^{-j}\cdot \pi)\eta(\pi(\cR)) .
$$
We have $\sigma_j\in H^s(\Gh)$ with 
\begin{equation}
\|\sigma_j\|_{H^s} 
\leq
 \|\sigma\|_{H^s_{l.u.,R},\cR,\eta}.
\label{eq_Hsnorm_sigmaj_Hslusigma}
\end{equation}
By Corollary \ref{cor_prop_Hslu_norm} \eqref{item_cor_prop_Hslu_norm_sob}, 
$\sigma$ and the $\sigma_j$'s are in $L^\infty(\Gh)$ 
and thus define Fourier multipliers
$$
T:\phi \mapsto \cF_G^{-1} \{ \sigma \widehat \phi\}
\quad\mbox{and}\quad
T_j:\phi \mapsto \cF_G^{-1} \{ \sigma_j \widehat \phi\},
$$
which are bounded on $L^2(G)$.
Their convolution kernels are respectively
$\kappa:=\cF_G^{-1}\sigma \in \cS'(G)$, 
and  $\kappa_j:=\cF_G^{-1}\sigma_j$ is in $L^2(G)$.

By Lemma \ref{lem_HsGh_sob_embedding}, 
the function $\kappa_j$ is integrable.

\begin{remark}
Even if it is not needed we can easily show that
$$
\int_G \kappa_j(x) dx
=0.
$$
Indeed denoting by $1_{\Gh}$ the trivial representation, we have
$$
\int_G \kappa_j(x) dx
=\widehat \kappa_j (1_{\Gh})
=\sigma_j(1_{\Gh})
=\sigma(2^{-j}\cdot 1_{\Gh}) \eta(1_{\Gh} (\cR)).
$$
Since the infinitesimal representation of $1_{\Gh}$ is identically zero 
and $\eta$ is supported away from 0, we have
$\eta(1_{\Gh} (\cR))=0$ and therefore the integral of $\kappa_j$ is zero.
\end{remark}

The sum $\sum_{j\in \bZ} \eta_j(\cR)$ 
converges towards the identity in the strong  
operator norm on  $L^2(G)$.
Formally 
we have 
$$
T = \sum_{j\in \bZ} T_ j\eta_j(\cR),
\quad
\sigma=\sum_{j\in \bZ} \sigma_j (2^j \pi),
\quad\mbox{and}\quad
\kappa =\sum_{j\in \bZ} 2^{-Qj}\kappa_j(2^{-j}\cdot).
$$

Let us prove that the last sum has a meaning and that the first Calder\'on-Zygmund condition is statisfied:
\begin{lemma}
\label{lem_kappa_L1loc}
The function $\kappa$ is locally integrable on $G\backslash\{0\}$.
Moreover the sum $\sum_{j\in \bZ} 2^{-Qj}\kappa_j(2^{-j}\cdot)$ 
converges to $\kappa$ in 
$L^1_{loc}(G\backslash\{0\})$.
\end{lemma}
\begin{proof}
Let $m\in \bZ$ be fixed.
By the change of variable given by the dilation $D_j$, 
we have for each $j\in \bZ$ that
$$
\int_{2^m\leq |x|\leq 2^{m+1}} 
|2^{-Qj}\kappa_j(2^{-j}x)| dx
=
\int_{2^{m-j}\leq |x|\leq 2^{m-j+1}} 
|\kappa_j(x)| dx.
$$

If $m-j\geq0$, we have
\begin{eqnarray*}
&&\int_{2^{m-j}\leq |x|\leq 2^{m-j+1}} 
|\kappa_j(x)| dx
=
\int_{2^{m-j}\leq |x|\leq 2^{m-j+1}} 
|\kappa_j(x)|  (1+|x|)^\epsilon (1+|x|)^{-\epsilon}  dx
\\
&&\qquad\lesssim
 2^{(m-j)(-\epsilon)} \|\kappa_j  (1+|\cdot|)^\epsilon\|_{L^1(G)}
\lesssim 
 2^{(j-m)\epsilon} 
\|\kappa_j  (1+|\cdot|)^s\|_{L^2(G)} ,
\end{eqnarray*}
by Corollary \ref{cor_prop_polar_coord}, 
as long as $s-\epsilon >Q/2$.

If $m-j<0$, we have by the Cauchy-Schwartz inequality that
\begin{eqnarray*}
&&\int_{2^{m-j}\leq |x|\leq 2^{m-j+1}} 
|\kappa_j(x)| dx
=
\int_G
(1+|x|)^s|\kappa_j(x)| 
(1+|x|)^{-s} 1_{2^{m-j}\leq |x|\leq 2^{m-j+1}}
dx
\\
&&\qquad\leq 
\|\kappa_j  (1+|\cdot|)^s\|_{L^2(G)}
\| (1+|\cdot|)^{-s} 1_{2^{m-j}\leq |x|\leq 2^{m-j+1}}\|_{L^2(G)}.
\end{eqnarray*}
Note that 
$$
\| (1+|\cdot|)^{-s} 1_{2^{m-j}\leq |x|\leq 2^{m-j+1}}\|_{L^2(G)}
\lesssim 2^{(m-j) \frac Q2},
$$
and that 
by \eqref{eq_Hsnorm_sigmaj_Hslusigma},
$$
\|\kappa_j  (1+|\cdot|)^s\|_{L^2(G)}
=\|\sigma_j\|_{H^s}
\leq
 \|\sigma\|_{H^s_{l.u.,R},\cR,\eta}.
$$

We choose $\epsilon = (s+Q/2)/2$.
We can now sum over $j\in \bZ$ to obtain
\begin{eqnarray*}
&&\sum_{j\in \bZ}
\int_{2^m\leq |x|\leq 2^{m+1}} 
|2^{-Qj}\kappa_j(2^{-j}x)| dx
=\sum_{j=-\infty}^{m-1}+\sum_{j=m}^\infty
\\&&\qquad\lesssim
\sum_{j=-\infty}^{m-1}
 2^{\frac Q2 (m-j)}  \|\sigma\|_{H^s_{l.u.,R},\cR,\eta}
+
\sum_{j=m}^\infty
 2^{(j-m)\epsilon} 
 \|\sigma\|_{H^s_{l.u.,R},\cR,\eta}
\\&&\qquad\lesssim
 \|\sigma\|_{H^s_{l.u.,R},\cR,\eta}.
\end{eqnarray*}
This implies that 
$\kappa =\sum_{j\in \bZ} 2^{-Qj}\kappa_j(2^{-j}\cdot)$
 is integrable on 
$\{2^{m}\leq |x|\leq 2^{m+1}\}$.
Therefore $\kappa$ is locally integrable on $G\backslash\{0\}$. 
\end{proof}

Let us show the Calder\'on-Zygmund inequality on the kernel:
\begin{lemma}
\label{lem_CZ}
Let us rewrite
$$
K(x,y)=\kappa (y^{-1}x)
\quad\mbox{and}\quad
d(x,y)=|y^{-1}x|.
$$

There exists $C>0$ such that for any two distinct points $y,y'\in G$
we have:
$$
\int_{d(x,y)>4c d(y,y')}
|K(x,y)-K(x,y')|dx \leq C \|\sigma\|_{H^s_{l.u.,R},\cR,\eta}.
$$

For $K_*(x,y)=\kappa^*(y^{-1}x)=\bar \kappa(x^{-1} y)$,
there exists $C>0$ such that for any two distinct points $y,y'\in G$
we have
$$
\int_{d(x,y)>4c d(y,y')}
|K_*(x,y)-K_*(x,y')|dx \leq C   \|\sigma\|_{H^s_{l.u.,L},\cR,\eta}.
$$
\end{lemma}
Here $c$ denotes the constant in the triangular inequality for the chosen quasi-norm, see Proposition \ref{prop_homogeneous_quasi_norm}.

\begin{proof}[Proof of Lemma \ref{lem_CZ}]
Let $y,y'\in G$ be two distinct points.
Let $h$ be the point $h:={y'}^{-1} y$ in $G\backslash\{0\}$
and let $m\in \bZ$ be the integer such that 
$2^m \leq 4c |h|<2^{m+1}$.
After the change of variable $z=y^{-1}x$ we see that
$$
\int_{d(x,y)>4c d(y,y')}
\!\!\!\!\!\!\!\!\!\!\!\! 
|K(x,y)-K(x,y')|dx 
=
\int_{|z|> 4c |h|}\!\!\! \!\!\! \!\!\! 
|\kappa(z)-\kappa(hz)| dz
\leq
\sum_{j\in \bZ} I_j,
$$
where
$$
I_j:=\int_{|z|> 4c |h|} |2^{-jQ}\kappa_j(2^{-j}z)-2^{-jQ}\kappa_j(2^{-j}(hz))| dz,
$$
since  $\kappa=\sum_j 2^{-jQ} \kappa (2^{-j}\cdot)$.
Using  the change of variable $2^{-j}z=w$, 
we have
\begin{eqnarray*}
I_j=
\int_{2^j|w|> 4c |h|}
\!\!\!\!\!\!\!\!\!\!\!\!
 |\kappa_j(w)-\kappa_j((2^{-j}h) w))| dw.
\end{eqnarray*}

If $j<m$
we use
\begin{eqnarray*}
I_j&\leq&
\int_{2^j|w|> 4c |h|} |\kappa_j(w)| dw
+
\int_{2^j|(2^{-j}h)^{-1}w'|> 4c |h|} |\kappa_j( w'))| dw',
\end{eqnarray*}
after the change of variable $w'=(2^{-j}h) w$.
The triangular inequality implies that
$$
2^j|(2^{-j}h)^{-1}w'|> 4c |h|
\Longrightarrow
|w'|> 3c 2^{-j}|h| \geq \frac 34 2^{-j+m}.
$$
Therefore,
$$
I_j\leq
2\int_{|w|> \frac 34 2^{-j+m}} |\kappa_j(w)| dw
\lesssim  2^{\epsilon(j-m)}\|(1+|\cdot|)^\epsilon \kappa_j \|_{L^1(G)}.
$$
 By Corollary \ref{cor_prop_polar_coord} and Lemma \ref{lem_HsGh_sob_embedding} with \eqref{eq_Hsnorm_sigmaj_Hslusigma},
 we have
$$
\|(1+|\cdot|)^\epsilon \kappa_j \|_{L^1(G)}
\lesssim
\|\sigma\|_{H^s_{l.u.},\cR,\eta}.
$$
So we have obtained in the case $j<m$,
$$
I_j\lesssim 2^{\epsilon(j-m)}\|\sigma\|_{H^s_{l.u.},\cR,\eta}.
$$

If $j\geq m$, 
we use the $L^1$-mean value theorem given in  Lemma \ref{lem_L1mv}:
$$
I_j
\lesssim 
\sum_{\ell=1}^n
|2^{-j}h|^{\upsilon_\ell}
\|\tilde X_\ell \kappa_j\|_{L^1(G)}
\lesssim 
2^{m-j}
\sum_{\ell=1}^n
\|\tilde X_\ell \kappa_j\|_{L^1(G)}.
$$
as $1\leq \upsilon_1\leq \ldots\leq \upsilon_n$.
 By Corollary \ref{cor_prop_polar_coord} and Lemma \ref{lem_HsGh_sob_embedding}, 
 we have
$$
\|\tilde X_\ell \kappa_j\|_{L^1(G)}
\lesssim
\|(\tilde X_\ell \kappa_j) (1+|\cdot|)^s\|_{L^2(G)}
$$
and by the Plancherel formula (see \eqref{eq_plancherel_formula})
\begin{eqnarray*}
\|(\tilde X_\ell \kappa_j) (1+|\cdot|)^s\|_{L^2(G)}
&=&
\|  \sigma_j(\pi) \pi(X_\ell)\|_{H^s(\Gh)}
\\
&=&
\|\sigma(2^{-j}\cdot \pi) \eta(\pi(\cR)) \ 
\omega(\pi(\cR)) \pi(X_\ell)\|_{H^s(\Gh)},
\end{eqnarray*}
where $\omega\in \cD(0,\infty)$ is identically equal to 1 on the support of $\eta$.
By Hulanicki's Theorem, cf. Theorem \ref{thm_hula},
the function 
$g_\ell :=\cF_G^{-1}\{\omega(\pi(\cR)) \pi(X_\ell)\}$
is Schwartz.
By \eqref{eq_cq_lem_L1omegas},
we have 
\begin{eqnarray*}
&&\|\sigma(2^{-j}\cdot \pi) \eta(\pi(\cR)) \ 
\omega(\pi(\cR)) \pi(X_\ell)\|_{H^s(\Gh)}
=
\|\sigma(2^{-j}\cdot \pi) \eta(\pi(\cR)) \ \widehat g_\ell(\pi)
\|_{H^s(\Gh)}
\\&&\qquad\lesssim
\|\sigma(2^{-j}\cdot \pi) \eta(\pi(\cR)) \|_{H^s(\Gh)}
\leq
 \|\sigma\|_{H^s_{l.u.,R},\cR,\eta}.
\end{eqnarray*}
So we have obtained in the case $j\geq m$ that
$$
I_j\lesssim 2^{m-j}  \|\sigma\|_{H^s_{l.u.,R},\cR,\eta}.
$$

We can now go back to
\begin{eqnarray*}
\sum_{j\in \bZ} I_j
\lesssim 
 \sum_{j<m} 2^{\epsilon(j-m)}\|\sigma\|_{H^s_{l.u.,R},\cR,\eta}+
 \sum_{j\geq m} 2^{m-j} \|\sigma\|_{H^s_{l.u.,R},\cR,\eta}
\lesssim
 \|\sigma\|_{H^s_{l.u.,R},\cR,\eta}.
\end{eqnarray*}

For $K_*$, 
after the change of variable $z=x^{-1}y$ and setting $h'=y^{-1}y'$,
we see that
$$
\int_{d(x,y)>4c d(y,y')}
|K_*(x,y)-K_*(x,y')|dx 
=
\int_{|z|> 4c |h|} |\kappa(z)-\kappa(zh')| dz.
$$
We proceed exactly in the same way as above using left invariant vector fields $X_\ell$.
\end{proof}

Hence the operator $T$ satisfies the hypotheses of 
 the Calder\'on-Zygmund theorem in the context of graded Lie groups,
 and more generally on spaces of homogeneous type cf. 
 \cite[Ch. III]{coifman+weiss}.
This implies Theorem \ref{thm_main} and the following 
proof of Corollary \ref{cor_thm_main}.

\begin{proof}[Proof of Corollary \ref{cor_thm_main}]
Part (1) follows from Corollary \ref{cor_prop_Hslu_norm}.
For Part (2), 
Lemmata \ref{lem_kappa_L1loc}  and \ref{lem_CZ} show that, 
if $\sigma\in H^s_{l.u.,R} $ for some $s>Q/2$, 
then $\kappa$ is a Calder\'on-Zygmund kernel, see 
\cite[Ch. III]{coifman+weiss} or \cite[\S3.2.3]{FR-monograph}.
We proceed in the same way for  Part (3):
using Lemma \ref{lem_Hslu_LR}:
if $\sigma\in H^s_{l.u.,L} $ for some $s>Q/2$, 
then $\kappa^*$ is a Calder\'on-Zygmund kernel.
As $T^*=T_{\sigma^*}$, this shows (3).
\end{proof}

\end{document}